\documentclass[11pt]{article}
 
\usepackage{amssymb}
\usepackage{amsmath}
\usepackage{amsthm}
\usepackage{geometry} 
\usepackage{hyperref}
\usepackage{parskip}
\usepackage{comment}
\usepackage[toc]{appendix}

\usepackage{amsfonts}
\usepackage{xcolor}
\usepackage{soul,xcolor}

\definecolor{ultramarine}{rgb}{0.07, 0.04, 0.56}  
\usepackage{chemformula}
 
\usepackage{lineno}

\usepackage{enumerate}

\definecolor{darkgreen}{rgb}{0,0.7,0}
\definecolor{darkred}{rgb}{0.5,0,0}
\definecolor{ultramarine}{rgb}{0.07, 0.04, 0.56}

\newtheorem{theorem}{Theorem}[section]
\newtheorem{proposition}[theorem]{Proposition}
\newtheorem{lemma}[theorem]{Lemma}

\theoremstyle{remark}
\newtheorem{remark}{Remark}
\newtheorem{definition}[theorem]{Definition}

\usepackage[
    backend=biber,
    style=numeric,
    sorting=nyt,
    maxbibnames=10,
  ]{biblatex}
  \renewbibmacro{in:}{%
  \ifentrytype{article}
    {}
    {\bibstring{in}%
     \printunit{\intitlepunct}}}
\usepackage{authblk}
\allowdisplaybreaks
\begin{document}

\title{A Feynman--Kac representation of a non-conservative and path-dependent nonlinear reaction-diffusion-advection system}

\author{Daniela Morale\thanks{Dept. of Mathematics, University of Milano, \href{mailto:daniela.morale@unimi.it}{daniela.morale@unimi.it}} , Leonardo Tarquini\thanks{Dept. of Mathematics, University of Oslo, \href{mailto:leonardt@math.uio.no}{leonardt@math.uio.no}}, Stefania Ugolini\thanks{Dept. of Mathematics, University of Milano, \href{mailto:stefania.ugolini@unimi.it}{stefania.ugolini@unimi.it}}}

\maketitle
\begin{abstract} 
    We provide a probabilistic interpretation of a weakly parabolic PDE--ODE system with a reaction term, which makes the dynamics non-conservative. As a consequence,   the solution is represented as the density of a sub-probability measure solving a Feynman--Kac-type equation, where the time-marginal law of the underlying process is weighted by a survival probability induced by the reaction. This leads to a coupled stochastic formulation consisting of a non-Markovian stochastic differential equation with path-dependent coefficients and the associated Feynman--Kac-type equation. We prove well-posedness of the resulting stochastic system. Finally, we introduce the corresponding interacting particle system and show that its empirical measure, suitably weighted by the survival probability associated with the reaction rate, converges to the limiting sub-probability.
\end{abstract}


  
\section{Introduction}
The aim of the present paper is to provide a probabilistic representation of the measure-valued solution of the following reaction-diffusion-advection partial differential equation (PDE)
\begin{equation}\label{eq:PDE_nu_intro}
    \partial_t\nu_t = \Delta\nu_t - \nabla\cdot\left[
    b\left(\int_0^t K*\nu_s\,ds,\int_0^t \nabla K*\nu_s\,ds\right)\nu_t\right]-\lambda c_0
    \exp\left(-\lambda\int_0^t K*\nu_s\,ds\right)\nu_t,
\end{equation}
with initial condition $\nu_0(dx)=\rho_0(x)dx$, where
\begin{equation}\label{eq:def_b_PDE}
    b(x,y) := -\varphi_1\lambda c_0\frac{\exp(-\lambda x)y}{\varphi_0+\varphi_1c_0 \exp(-\lambda y)}\quad \forall x,y\in \mathbb{R},
\end{equation}
$c_0,\lambda,\varphi_0,\varphi_1$ are constant parameters, and $K$ is a smooth kernel.

While the model \eqref{eq:PDE_nu_intro}--\eqref{eq:def_b_PDE} arises from the real-world application described in Section \ref{sec:motivation}, the contribution of the present paper lies in its probabilistic interpretation and the analysis of the resulting stochastic system.

The PDE is path--dependent and non-conservative, that is, the total concentration is not conserved, because of the reaction term; hence, even when the initial measure is a probability measure, $\nu_t$ is a sub-probability measure for any $t \in (0,T]$. This non-conservative feature prevents from interpreting equation \eqref{eq:PDE_nu_intro} as a nonlinear Fokker--Planck equation describing the evolution of the time marginal laws of the corresponding stochastic process, solving an SDE whose drift coincides with the transport term appearing in the PDE. Indeed, we interpret the measure-valued solution $\nu_t$ of the PDE \eqref{eq:PDE_nu_intro} as the sub-probability measure solution of the following  nonlinear Feynman--Kac-type equation
\begin{equation}\label{eq:def_nu_intro}
    \nu_t(f) = \mathbb{E}\left[f(Y_t)\exp\left(-\lambda c_0\int_0^t\exp\left(-\lambda\int_0^s (K*\nu_r)(Y_r),dr\right)ds\right)\right],
\end{equation}
for every $t\in[0,T]$ and every Schwartz function $f\in \mathcal{S}(\mathbb{R})$. The  stochastic  process $Y$ is the solution of the associated stochastic differential equation (SDE)
\begin{equation}\label{eq:SDE_con_nu_intro}
    \begin{aligned}
    &Y_t = Y_0 + \int_0^t b\left(\int_0^s (K*\nu_r)(Y_s)dr,\int_0^s (\nabla K*\nu_r)(Y_s)dr\right)ds + \sqrt{2}W_t, \quad t\in[0,T],\\
    &Y_0 \sim \rho_0(x)dx.
    \end{aligned}
\end{equation}

The resulting formulation can be viewed as a Feynman--Kac representation in which the law of the process $Y$ is weighted by a discount factor accounting for the mass loss induced by the reaction term.  Hence, we refer to $\nu_t$ as the Feynman--Kac weighted law of $Y_t$.
 
A different but equivalent approach has been developed in \cite{2025_morale_tarquini_ugolini} via McKean--Vlasov SDEs subject to a killing mechanism.  There, the loss of mass induced by the reaction term is represented through a stopping time $\tau$ and the sub-probability measure solution to \eqref{eq:PDE_nu_intro} is represented as $\nu_t:=\mathbb{P}(Y_t\in\cdot,\tau>t)$. For a discussion on the equivalence of the two approaches, see \cite{morale2026pathdependentmckeanpdesreaction}.

The main contribution of the present paper is to develop the probabilistic analysis of the nonlinear McKean--Feynman--Kac system \eqref{eq:SDE_con_nu_intro}-\eqref{eq:def_nu_intro} and of its particle approximation. More precisely, we show the existence of the sub-probability measure flow $\{\nu_t\}_{t\in[0,T]}$ defined by the Feynman--Kac-type equation \eqref{eq:def_nu_intro}, and we establish that it provides a distributional solution to the PDE \eqref{eq:PDE_nu_intro}. We then prove strong existence and pathwise uniqueness of the solution to the path-dependent McKean--Feynman--Kac SDE \eqref{eq:SDE_con_nu_intro}, whose drift depends on the past evolution of the regularised field generated by $\{\nu_t\}_{t\in[0,T]}$. Finally, we introduce the corresponding mean-field interacting
particle system and prove convergence of the empirical density with Feynman--Kac-type weights by establishing a propagation of chaos result.

There is a long-standing literature on probabilistic interpretations of PDEs in the conservative case; see, for instance,
\cite{2023_Tomasevic_Fournier,2020_tomasevic_talay,2025_GM,2000_stevens,
1990_Oelschlaeger}. The extension of probabilistic representation techniques
for McKean-type equations to a large class of non-conservative PDEs was
developed in
\cite{2017_russo_barbu,2016_Russo,2019_russo_proceedings}. There, the standard probabilistic representation of Fokker--Planck equations is extended to non-conservative nonlinear Fokker--Planck equations by introducing a Feynman--Kac representation of the solution. The present paper extends this framework to the path-dependent setting by introducing a suitable Banach space and norm that accounts both for the survival probability induced by the reaction term and for the integral operator responsible for the non-Markovian character of the drift and discount factor relative to the reaction coefficient. Moreover, particular care is required due to the presence of the gradient of the density in the drift term, especially in the analysis of the regularity properties of the solution to the Feynman--Kac-type equation and in establishing the convergence of the empirical density of the associated particle system.

The paper is organized as follows. Section \ref{sec:motivation} introduces the weakly parabolic PDE--ODE system underlying our analysis, derives the PDE studied in this paper, and discusses its real-world application. Section \ref{sec:prob_interpretation} contains the core of our probabilistic interpretation, where we show that the solution of the Feynman--Kac-type equation \eqref{eq:def_nu_intro} also solves the measure-valued PDE \eqref{eq:PDE_nu_intro}. Section \ref{sect:well_posedness_reg_Feynman-Kac} is devoted to establishing the well-posedness of this equation in view of its path dependence. Building on continuity properties of its solution, Section \ref{sec:Existence_uniqueness_SDE} proves strong existence and uniqueness for the stochastic model \eqref{eq:SDE_con_nu_intro},\eqref{eq:def_nu_intro}. In Section \ref{sec:propagation_chaos}, we introduce the associated particle system and establish propagation of chaos. Finally, Appendix \ref{appendix} collects the proofs of the more technical results.

\paragraph{Notation.}
Let
$(\Omega,\mathcal{F},\mathbb{F}=\{\mathcal{F}_t\}_{t\in[0,T]},\mathbb{P})$
be a filtered probability space. All stochastic processes throughout the paper
are assumed to be defined on this space and adapted to $\mathbb{F}$.

For any Banach space $V$, we denote by $\mathcal{B}(V)$ its Borel
$\sigma$-algebra.

$\mathcal{S}(\mathbb{R})$ denotes the Schwartz space and it is defined as
\begin{equation}\label{eq:def_schwartz_space}
    \mathcal{S}(\mathbb{R}) := \left\{\varphi \in C^\infty(\mathbb{R}) \;:\;\sup_{x\in\mathbb{R}} \left| x^k\varphi^{(n)}(x) \right| < \infty \quad \forall k,n \in \mathbb{N}_0\right\}.
\end{equation}

\section{Motivation}\label{sec:motivation}
We begin by introducing the following weakly parabolic PDE--ODE system, from which the PDE model \eqref{eq:PDE_nu_intro} investigated in this work can be derived.
\begin{equation} \label{eq:modello_Natalini} 
	\begin{cases}
		\partial_t \rho = \nabla\cdot(\varphi(c)\nabla s ) -\lambda\rho c;   \\
		\partial_t c = -\lambda \rho c,  
	\end{cases}\quad (t,x)\in (0,T]\times\mathbb{R},
\end{equation}
where $T>0$, subject to the initial conditions
\begin{equation} \label{eq:modello_Natalini_condition1} 
	\begin{cases}
	    \rho(0,x)= \rho_0(x);\\
        c(0,x) = c_0,
	\end{cases}\quad x\in\mathbb{R}.
\end{equation}

The model \eqref{eq:modello_Natalini} arises in the mathematical description of marble sulphation \cite{2004_ADN,2023_Bonetti_natalini_NLA}, a specific chemical reaction caused by pollution which affects marble surfaces. In \eqref{eq:modello_Natalini}, $c$ denotes the density of calcite, $s$ the porous concentration of sulphur dioxide, and $\rho$ its density in the material. The two quantities are related by $\rho=\varphi(c)s$, where $\varphi(c)$ denotes the porosity of the material.

In system \eqref{eq:modello_Natalini}, $\rho = \varphi(c) s$, $\lambda\in \mathbb{R}_{+}$, and $\varphi(c)=\varphi_0+\varphi_1 c$, where $\varphi_0\in \mathbb R_+$, $\varphi_1 \in \mathbb R$ are such that $\varphi_0,\varphi_0+\varphi_1 c_0 \in (0,\bar{\varphi})$ with $\bar{\varphi}\in\mathbb{R}_+$.

For the sake of simplicity, in the present work   we consider a constant initial condition for $c$ in \eqref{eq:modello_Natalini_condition1}; however, the system is well-posed also in the case $c_0(x)\in (0,C_0)$, such that $C_0-c_0 \in H^1(\mathbb{R})$. As a consequence, for any $(t,x)\in[0,T]\times \mathbb R$, one has that $c(t,x)\in (0,C_0)$ \cite{2005_GN_NLA}. Furthermore, for the well-posedness of the system, it is also required that the initial condition $\rho_0(x) \in (0,S_0)$ and that $\rho_0\in L^2(\mathbb R)$ \cite{2005_GN_NLA,2007_GN_CPDE}. Now, since it is shown that $\rho(t,x) \in (0,S_0)$ for any $(t,x)\in [0,T]\times\mathbb R$, whenever $S_0\le 1$, concentration  $\rho$  is a map from $[0,T]\times \mathbb{R}$ to $[0,1]$. This is the case we consider here.
 
In recent decades, such a system has been well-studied in the deterministic framework both from the analytical and from the numerical point of view \cite{2007_AFNT_TPM,2007_AFNT2_TPM,2004_ADN,2023_Bonetti_natalini_NLA,2019_BCFGN_CPAA,2005_GN_NLA,2007_GN_CPDE}. However, only recently a probabilistic interpretation of solutions as mean-field limits of interacting particles has been investigated. In particular, \cite{2025_morale_tarquini_ugolini} introduced a probabilistic representation via McKean--Vlasov SDEs with killing. Here, we propose an alternative yet equivalent  \cite{morale2026pathdependentmckeanpdesreaction} probabilistic approach, based on a different microscopic description of the particle dynamics.

To obtain \eqref{eq:PDE_nu_intro}, we rewrite system \eqref{eq:modello_Natalini} as a single PDE for the concentration $\rho$ coupled with an explicit function for $c$. To make this precise, let us notice that \eqref{eq:modello_Natalini} may be written in terms of $\rho$ as
\begin{equation}\label{eq:modello_Natalini_riscritto}
    \partial_t\rho = \Delta\rho - \nabla \cdot \left( \frac{\nabla\varphi(c)}{\varphi(c)}\rho \right) - \lambda c \rho.
\end{equation}
We  denote by   $\rho(\cdot,x)(t)$ the integral function of the concentration $\rho$, i.e. 
$$\rho(\cdot,x)(t):=\int_0^t \rho(s,x)ds.$$

For any $(t,x)\in [0,T]\times\mathbb{R}$, the velocity field that the concentration $\rho$ moves with becomes
\begin{equation*}
    \frac{\nabla\varphi(c)}{\varphi(c)}=b\big(\rho(\cdot,x)(t), \nabla \rho(\cdot,x)(t) \big),
\end{equation*}
where $b$ is defined in \eqref{eq:def_b_PDE}, obtained by substituting the following explicit solution of the equation for $c$ as a function of $\rho$:
\begin{equation}\label{eq:c_explicit}
    c(t,x) = c_0\exp\left(-\lambda\int_0^t\rho(u,x)du \right)= c_0\exp\big( -\lambda  \rho(\cdot,x)(t) \big).
\end{equation}
Therefore, system \eqref{eq:modello_Natalini} is equivalent to the following evolution equation for $\rho$ 
\begin{equation}\label{eq:PDE_rho_only}
        \partial_t\rho (t,x)  = \Delta\rho(t,x)   - \nabla \cdot \left( b\big(\rho(\cdot,x)(t), \nabla \rho(\cdot,x)(t) \big)\rho(t,x) \right) - \lambda c_0\exp\big( -\lambda  \rho(\cdot,x)(t) \big) \rho(t,x),
\end{equation}
subject to the initial condition $\rho(0,\cdot) = \rho_0(\cdot)$, where the advection coefficient $b$ is defined in \eqref{eq:def_b_PDE}.

Since the analysis is carried out in the case of a McKean--Vlasov-type drift, the drift of the associated SDE is regularised by introducing a non-local dependence upon $\rho$. Since the advection term in the classical Fokker-Planck equation is related to the drift of the SDE, from \eqref{eq:c_explicit}, the regularisation in the original model \eqref{eq:modello_Natalini} has to be done at the level of the ODE for $c$. This is the underlying reason why we consider a regularisation of the ODE in \eqref{eq:modello_Natalini} by introducing a non-local dependence of $c$ upon $\rho$:
\begin{equation*}
    \partial_t c = -\lambda   (K*\rho) c,
\end{equation*}
where $*$ stands for the convolution with a smooth kernel  $K: \mathbb{R} \rightarrow \mathbb{R}$, i.e.  a $C^{\infty}(\mathbb{R})$ function with compact support or with exponential decay to zero such that $\int_{\mathbb{R}} K(x)dx = 1$.

As a consequence, for any $(t,x)\in [0,T]\times \mathbb R$, equation \eqref{eq:PDE_rho_only} becomes
\begin{equation}\label{eq:Natalini_density}
    \partial_t\rho(t,x) = \Delta\rho(t,x) - \nabla \cdot \left[ b\big(K* \rho(\cdot,x)(t), \nabla K* \rho(\cdot,x)(t)\big) \rho\, \right] -\lambda c_0\exp\big( -\lambda  K*\rho(\cdot,x)(t) \big) \rho(t,x),
\end{equation}
and equation \eqref{eq:PDE_nu_intro} reads as its measure-valued version.

Finally, to the best of our knowledge, pure statistical models \cite{2018_saba}, or deterministic partial differential equation-based models cited above, have been proposed to study the dynamical evolution of degradation phenomena. 
Some stochastic approaches to the degradation modelling have recently been presented \cite{MACH2021_AGMMU,2025_MauMorUgo,2025_Mach2023_MRU}. At the macroscale, an extension of the model \eqref{eq:modello_Natalini} through a coupling with a source of randomness via a stochastic dynamical boundary condition has been provided in \cite{2025_MACH2023_AMMU,2025_MauMorUgo}, making system \eqref{eq:modello_Natalini} a random PDE. The qualitative behaviour of the solutions to such a stochastic boundary value problem has been investigated from a numerical point of view in \cite{2024_ArceciMoraleUgolini_arxiv}, and some rigorous convergence results of some related discrete schemes in a discrete Besov space framework have been established in \cite{2025_ADMU_arxiv}. On the other hand, at a nanoscale, a first stochastic interacting particle model for the dynamics of the calcium carbonate, sulphur dioxide and gypsum molecules has been proposed in \cite{2026_JMMRU,2025_Mach2023_MRU}.

\section{The probabilistic interpretation of the PDE model}\label{sec:prob_interpretation}
By means of a formal application of It\^{o}'s formula, we show that solutions to the Feynman--Kac-type equation \eqref{eq:def_nu_intro} yield measure-valued solutions to the corresponding PDE system \eqref{eq:PDE_nu_intro}.

Throughout the remainder of the paper, we adopt the following notion of smooth mollifier.
\begin{definition}\label{def:kernel}
     A kernel $K: \mathbb{R} \rightarrow \mathbb{R}$ is said a \emph{smooth mollifier} if it satisfies the following properties
\begin{enumerate}[i)]
    \item there exist two constants $M_K,M_K^\prime\in \mathbb R_+$ such that, for any $y\in\mathbb{R}$,  \begin{equation}\label{eq:K_MK}
     |K(y)|\leq M_K, \quad  |\nabla K(y)|\leq M_K^\prime;   
    \end{equation}
    \item there exists a constant $L_K \in \mathbb R_+$  such that, for any $y,y'\in\mathbb{R}$, \begin{equation}\label{eq:K_LK}|K(y')-K(y)|\,\leq\,L_K|y'-y|;
    \end{equation}
    \item it is a probability density; therefore, $\int_{\mathbb{R}}K(y)\,dy =1$.
    \item $K\in L_1$ and there exists a constant $F_K$ such that
    \begin{equation}\label{eq:K_FK}
        \|\mathcal{F}(K)\|_1\,\leq\,F_K,
    \end{equation}
    where $\mathcal{F}(K)$ denotes the Fourier transform of $K$.
\end{enumerate}
\end{definition} 
\begin{remark}
    An example of smooth mollifier which satisfies the assumptions of Definition \ref{def:kernel} is the Gaussian kernel
    \begin{equation*}
        G(x,\sigma):=\frac{1}{\sqrt{2\pi\sigma^2}}\exp\left(-\frac{x^2}{2\sigma^2}\right),\quad \sigma>0.
    \end{equation*}
    Moreover, for every $\sigma>0$, $G(\cdot,\sigma)\in\mathcal{S}(\mathbb{R})$, where $\mathcal{S}(\mathbb{R})$ is the Schwartz space defined in \eqref{eq:def_schwartz_space}.
\end{remark}

We have the following result.  
\begin{theorem}\label{theo:prob_interpretation}
    Let $K$ be a smooth mollifier as in Definition \ref{def:kernel} and let $\{\nu_t\}_{t\in[0,T]}$ be a solution to the Feynman--Kac-type equation
    \begin{equation}\label{eq:def_nu}
        \nu_t(f) = \mathbb{E}\left[f(Y_t)\exp\left(-\lambda c_0\int_0^t\exp\left(-\lambda\int_0^s (K*\nu_r)(Y_s)\,dr\right)ds\right)\right],
    \end{equation}
    for every $t\in[0,T]$ and every $f\in \mathcal{S}(\mathbb{R})$, where the process $Y$ is solution to
    \begin{equation}\label{eq:SDE_con_nu}
        \begin{aligned}
            &Y_t = Y_0 + \int_0^t b\left(\int_0^s (K*\nu_r)(Y_s)dr,\int_0^s (\nabla K*\nu_r)(Y_s)dr\right)ds + \sqrt{2}W_t;\quad t\in[0,T]\\
            &Y_0\sim\rho_0(x)dx.
        \end{aligned}
    \end{equation}
    Then, $\{\nu_t\}_{t\in[0,T]}$ is characterised as the solution of the following measure-valued partial integro-differential equation 
    \begin{equation}\label{eq:PDE_nu}
    \partial_t\nu_t = \Delta\nu_t - \nabla\cdot\left[
    b\left(\int_0^t K*\nu_s\,ds,\int_0^t \nabla K*\nu_s\,ds\right)\nu_t\right]-\lambda c_0
    \exp\left(-\lambda\int_0^t K*\nu_s\,ds\right)\nu_t,
\end{equation}
    with initial condition $\nu_0(dx) = \rho_0(dx)$.
\end{theorem}
\begin{proof}
    For the sake of readability, we introduce the following notation. For every $s\in[0,T]$, let
    \begin{equation*}
        b_s := b\left(\int_0^s (K*\nu_r)(Y_s)dr,\int_0^s (\nabla K*\nu_r)(Y_s)dr\right)
    \end{equation*}
    and
    \begin{equation*}
        E_s := \exp\left(-\lambda c_0\int_0^s\exp\left(-\lambda\int_0^r (K*\nu_u)(Y_r)du\right)dr\right).
    \end{equation*}
    Notice that the process $E$ has finite variation; therefore, the quadratic covariation of $f(Y_s)$ and $E_s$ is zero for every $s\in[0,T]$.
    
    By applying It\^{o}'s lemma to $f(Y_t)E_t$,
    \begin{align*}
        d(f(Y_t)E_t) &= df(Y_s)E_s + f(Y_s)dE_s\\
        &\begin{aligned}
            \,\,=\,&\left(b_s\nabla f(Y_s)E_s + \Delta f(Y_s)E_s\right)ds + \sqrt{2}\nabla f(Y_s)E_sdW_s\\
            &-\lambda c_0 f(Y_s)E_s\exp\left(-\lambda\int_0^s (K*\nu_r)(Y_s)dr\right)ds
        \end{aligned}
    \end{align*}

    Taking the expectation and by applying Tonelli's theorem,
    \begin{equation*}
        \begin{aligned}
            \mathbb{E}[f(Y_t)E_t] = &\mathbb{E}[f(Y_0)E_0] + \int_0^t\mathbb{E}\left[b_s\nabla f(Y_s)E_s + \Delta f(Y_s)E_s\right]ds\\
            &-\lambda c_0\int_0^t\mathbb{E}\left[f(Y_s)E_s\exp\left(-\lambda\int_0^s (K*\nu_r)(Y_s)dr\right)\right]ds
        \end{aligned}
    \end{equation*}

    By equation \eqref{eq:def_nu},
    \begin{equation*}
        \begin{aligned}
            \int_{\mathbb{R}}f(x)\nu_t(dx) = &\int_{\mathbb{R}}f(x)\rho_0(dx) + \int_0^t\int_{\mathbb{R}}b\left(\int_0^s (K*\nu_r)(x)dr,\int_0^s (\nabla K*\nu_r)(x)dr\right)\nabla f(x)\nu_s(dx)\\
            &+\int_0^t\int_{\mathbb{R}}\Delta f(x)\nu_s(dx)ds - \lambda c_0\int_0^t\int_{\mathbb{R}}f(x)\exp\left(-\lambda\int_0^s (K*\nu_r)(x)dr\right)\nu_s(dx)ds,
        \end{aligned}
    \end{equation*}
    which reads as PDE \eqref{eq:PDE_nu} in a measure-valued sense.
\end{proof}

To further detail the probabilistic interpretation given in Theorem \ref{theo:prob_interpretation}, let us give the following definition. 
\begin{definition}\label{def:mollified_density}
  Let $K$ be a smooth mollifier as in Definition \ref{def:kernel} and $\{\nu_t\}_{t\in[0,T]}$ be a Borel measure solution of the measure-valued equation \eqref{eq:PDE_nu}. The function $u_K$ defined, for any $(t,x)\in [0,T]\times \mathbb R$, as  
  \begin{equation}\label{eq:def_mollified_density}
      u_K(t,x)=K*\nu_t(x)
  \end{equation} is called the \emph{mollified density} associated to $\{\nu_t\}_{t\in[0,T]}$.
\end{definition}

Proposition \ref{theo:prob_interpretation} provides a probabilistic interpretation of the regularised version of the deterministic model \eqref{eq:modello_Natalini}. Whenever the density $\{\rho(t,\cdot)\}_{t\in[0,T]}$ satisfies equation \eqref{eq:modello_Natalini}, it may be regarded as the possible density of the measure-valued solution $\{\nu_t\}_{t\in[0,T]}$ of the following  PDE 
\begin{equation}\label{eq:PDE_nu_u_K}
    \partial_t\nu_t  = \Delta\nu_t -\nabla\cdot\left[ b\left(\int_0^t u_K ds,\int_0^t \nabla  u_K ds\right)\nu_t\right] - \lambda c_0\exp\left( -\lambda \int_0^t  u_K ds \right)\nu_t.
\end{equation}

Such equation is driven by the mollified density \eqref{eq:def_mollified_density}, i.e., the averaged non-local function $u_K$, 
\begin{equation}\label{eq:mollified_equation}
    u_K(t,y)= \mathbb E\left[K(y-Y_t) \exp\left(-\lambda\,c_0 \int_0^t \exp\left( -\lambda u_K(\cdot,Y_s)(s)  \right) ds\right) \right],
\end{equation}
discounted by the survival probability according to the specific reaction term along the solution of the McKean--Vlasov SDE
\begin{equation}\label{eq:SDE_u_K}
        Y_t = Y_0 + \int_0^t b\left(u_K(\cdot,Y_s)(s),\nabla  u_K(\cdot,Y_s)(s)\right)ds + \sqrt{2}W_t,
\end{equation}
with the initial condition $u_K(0,x)=K*\rho_0(x)$.
 
In conclusion, at the microscale,  for any  mollifier $K$, the stochastic process $Y$ evolves according to \eqref{eq:SDE_u_K} under the influence of the nonlocal, discounted, and averaged field $u_K$, acting at the macroscale. The unique measure $\nu$  associated with $u_K$, as defined in Proposition \ref{theo:prob_interpretation}, evolves according to the PDE  \eqref{eq:PDE_nu_u_K}. For any $t\in [0,T]$, the measure  $\nu_t$ is a macroscopic observable at time $t$, while $u_K(t,\cdot)$ describes a nonlocal average of this quantity obtained through convolution with $K$, thereby introducing further non-local effects. In particular, $u_K$  depends on the distribution of the measure over the entire domain, rather than solely on the local behaviour around any point $x\in \mathbb R$. At the microscale, this macroscopic field governs the dynamics of a typical random particle given by $Y$. This interpretation is made rigorous through the propagation of chaos result established for the associated particle system (see Section \ref{sec:propagation_chaos}). Hence, we can regard the pair $(K,m)$ at the microscale as equivalent to the corresponding  macroscopic pairs $(K,u_K)$ and $(K,\nu)$ which represent, given $K$, the density of the chemical with and without nonlocal effects, respectively.

The coupled system \eqref{eq:SDE_u_K},\eqref{eq:mollified_equation} will be referred to as the \emph{McKean--Feynman--Kac SDE (MKFK-SDE)}. We now turn to the analysis of its well-posedness. We begin by establishing well-posedness results for the associated Feynman--Kac-type equation in the more general framework of an arbitrary measure, not necessarily corresponding to the law of a solution of \eqref{eq:SDE_u_K}.

\section{The Feynman--Kac-type equation: study of the well-posedness } \label{sect:well_posedness_reg_Feynman-Kac}
Let $C=C\left( [0,T];\mathbb{R}\right)$ and $C^+=C\left( [0,T];\mathbb{R}^{+} \right)$ be the space of the real and real positive continuous functions from $[0,T]$, respectively. Furhermore, let $\mathcal{P}\left(C\right)$ be  the space of probability measures  on $C$. For any $m\in\mathcal{P}\left(C\right)$, we denote by    $X^m$ the associated canonical process.

For every $t\in[0,T]$ and $z\in C^+$, we consider the following notation for the integral function
\begin{equation}\label{def:integral_function}
    z(t)=\int_0^t z_s ds,
\end{equation}
and we introduce the following auxiliary functions
\begin{eqnarray}\label{eq:Lambda_Feynman--Kac}
    \Lambda_t(z)&:=&\Lambda(t,z(t)):=-\lambda\,c_0\exp\left( -\lambda \int_0^t z_s\,ds \right)=  -\lambda\,c_0\exp\left( -\lambda z(t) \right)\,,\\\label{eq:V_in_FK}
    V_t\left(z  \right)   &:=&   \exp\left(  \int_0^t\Lambda\left(s,z(s)\right) ds\right).
\end{eqnarray}
While function $V_t$ takes into account the presence of a discount term, function $\Lambda_t$  takes into account the dependence upon the past history.

Now, given a smooth mollifier $K$, we write the regularised Feynman--Kac-type equation \eqref{eq:mollified_equation} for a generic measure over path space $m\in\mathcal{P}\left(C\right)$. We use the superscript $m$ to emphasize that the construction is performed with respect to a generic measure on path space, not necessarily corresponding to the law of the process $Y$ solving the SDE \eqref{eq:SDE_u_K}. This is
\begin{equation}\label{eq:Feynman_Kac_generic_measure}
    \begin{split}
        u^m_K(t,y) &= \int_{C} K(y - X_t^m(\omega))\,V_t\left( u^m_K\left(\cdot,X^m_\cdot(\omega)\right)\right)m(d\omega)\\
        &=\int_{C} K(y - X_t^m(\omega))\,\exp\left(  \int_0^t\Lambda\Big(s,u^m_K\left(\cdot,X^m_s(\omega)\right)(s)\Big) ds \right)m(d\omega),
    \end{split}
\end{equation} 
where  $\Lambda$ and $V_t$ are given by \eqref{eq:Lambda_Feynman--Kac} and \eqref{eq:V_in_FK}, respectively.

Before proving that equation \eqref{eq:Feynman_Kac_generic_measure} admits a unique solution, we show some continuity properties for the terms $\Lambda$ and $V$.
\begin{lemma}\label{Lemma:Lambda}
    For any $t\in[0,T]$ and $z,z^\prime \in C^+$, the function $\Lambda$ defined in \eqref{eq:Lambda_Feynman--Kac} is such that
    \begin{enumerate}[i)]
        \item $  \big|\Lambda(t,z(t))\big|\le \lambda c_0$, i.e.  it is uniformly bounded;
        \item $\big| \Lambda(t,z(t)) - \Lambda(t,z^\prime(t)) \big| \,\leq\, \lambda^2\,c_0|z(t) - z^\prime(t)| \,\leq\, \lambda^2\,c_0\int_0^t|z_s - z^\prime_s|\,ds$.
    \end{enumerate}
\end{lemma}
\begin{proof}
   Recalling the notation \eqref{def:integral_function} for integral functions, since $z\in C^+$ and $\lambda\in \mathbb R_+$, from the fact that $\exp\left( -\lambda z(t) \right)=\exp\left( -\lambda \int_0^t z_s\,ds \right)\leq 1$, it derives that
    \begin{equation*}
        \big| \Lambda(t,z(t)) \big| = \lambda\,c_0\,\exp\left( -\lambda \int_0^t z_s\,ds \right) \,\leq\, \lambda\,c_0\,.
    \end{equation*}
    As the   point ii) concerns, let us notice that for any $a,b\in\mathbb{R}$,
    \begin{equation}\label{ineq_aux_exp}
        e^b - e^a = (b-a)\int_0^1 e^{bx + (1-x)a}\,dx \,\leq\,|a - b|  e^{a \vee b  }.
    \end{equation}
    By taking $a:= -\lambda z(t), b:= -\lambda z^\prime(t)$ so that $a,b <0$ and $ e^{a \vee b  }\leq 1$, from  \eqref{eq:Lambda_Feynman--Kac} and \eqref{ineq_aux_exp},
    \begin{align*}
        \big| \Lambda(t,z(t)) - \Lambda(t,z^\prime(t)) \big| \,&=\, \lambda\,c_0\Bigg| \exp\left( -\lambda \int_0^t z_s\,ds \right) - \exp\left( -\lambda \int_0^t z^\prime_s\,ds \right) \Bigg|&&\\
        &\leq\,\lambda^2c_0 \Bigg| \int_0^t z_s\,ds - \int_0^t z^\prime_s\,ds \Bigg|\,\leq\,\lambda^2\,c_0\int_0^t |z_s - z^\prime_s|\,ds.&&
    \end{align*}
\end{proof}

 As a consequence of  Lemma \ref{Lemma:Lambda}, function $V_t$ in \eqref{eq:V_in_FK} shares similar properties with function $\Lambda$.
\begin{lemma}\label{Lemma:V_continuity}
    For every $t\in[0,T]$ and $z,z^\prime\in C^+$, the function $V_t$  defined in \eqref{eq:V_in_FK}  is such that
    \begin{enumerate}[i)]
        \item $0\le V_t(z)\le 1$;
        \item  $
        \big| V_t(z) - V_t(z^\prime) \big| \,\leq\, \lambda^2\,c_0\int_0^t|z(s) - z^\prime(s)|\,ds \,\leq\, \lambda^2\,c_0\int_0^t\int_0^s|z_r - z^\prime_r|\,drds.
    $
    \end{enumerate}
\end{lemma}
\begin{proof}
    Since $z\in C^+$, then it is bounded in $[0,T]$. Let $\overline{M}$ be its maximum. For any $t\in[0,T]$ and $z\in C^+$, since
    \begin{equation*}
        \exp\left( -\lambda \int_0^t z(s)\,ds \right) \,\geq\,  e^{-\lambda\,\overline{M}\,T} \,=:\,C(T)\,>\, 0,
    \end{equation*}
   we get that
    \begin{equation*}
        0 < V_t(z)   = \exp\left(-\lambda\,c_0 \int_0^t \exp\left( -\lambda \int_0^s z_r\, dr \right) ds\right) \,\leq\, e^{-\lambda\,c_0\,C(T)\,t} \,\leq\, 1.
    \end{equation*} 
By applying inequality \eqref{ineq_aux_exp} with $a\,:=\Lambda(t,z(t))<0$ and $b\,:=\Lambda(t,z^\prime(t))<0$, by ii) of Lemma \ref{Lemma:Lambda}, we conclude. Indeed, 
\begin{eqnarray*}
   \left| V_t(z) - V_t(z^\prime) \right| &=& \left| \int_0^t \Lambda(s,z(s)) ds - \int_0^t\Lambda(s,z^\prime(s))ds \right| \leq   \int_0^t\left| \Lambda(s,z(s))   -  \Lambda(s,z^\prime(s)) \right| ds \\&\leq\,&\lambda^2\,c_0\int_0^t\int_0^s|z_r - z^\prime_r|\,dr\,ds.
 \end{eqnarray*}
\end{proof}

We are now ready to prove the result of well-posedness for equation \eqref{eq:Feynman_Kac_generic_measure}.

\subsection{Existence and  uniqueness}
To start, we consider possible norms on the space of real-valued continuous processes $Z=\left(Z_t\right)_{t\in [0,T]}$ defined upon the canonical space $C=C\left( [0,T];\mathbb{R}\right)$. 
\begin{definition} \label{definition_c1}
    Fix $m\in\mathcal{P}\left(C\right)$. Let $C_1=\left(C_1,\| \cdot\|_{\infty,1}\right)$ be the Banach space of real-valued continuous processes  $Z=\{Z_t\}_{t\in [0,T]}$, endowed with the norm
    \begin{equation}\label{eq:norm_infty_1}
        \lVert Z \rVert_{\infty,1} := \mathbb{E}_m\left[ \sup_{t\leq T} |Z_t| \right] =  \int_{C} \sup_{t \leq T} |Z_t(\omega)| m(d\omega) < \infty.
    \end{equation}
\end{definition}
On the space $C_1$, one may introduce an equivalent norm that takes into account a possible  discount term. 
\begin{lemma}\label{lemma:norm_killing_equivalent}
    For any $M\geq 0$ and $Z\in C_1$, let us define the norm
    \begin{equation}\label{eq:norm_infty_1_M}
        \lVert Z \rVert_{\infty,1}^M \,:=\, \mathbb{E}_m\left[ \sup_{t\leq T} e^{-Mt} |Z_t| \right].
    \end{equation}
    Then, for any $M\geq 0$, the norms  $\lVert \cdot \rVert_{\infty,1}$ and  $\lVert \cdot \rVert_{\infty,1}^M$ are equivalent. 
    Furthermore, for any $M\geq0$, the space $C_{1,M}=\left(C,\lVert \cdot \rVert_{\infty,1}^M\right)$ is a Banach space.
\end{lemma}
\begin{proof}
    On the one hand,
    \begin{equation*}
         \lVert Z \rVert_{\infty,1}^M = \mathbb{E}_m\left[ \sup_{t\leq T} e^{-Mt}\,\big|Z_t\big| \right] \,\leq\, \left(\sup_{t\leq T} e^{-Mt}\right)\mathbb{E}_m\left[ \sup_{t\leq T} \big|Z_t\big| \right] = \mathbb{E}_m\left[ \sup_{t\leq T} \big|Z_t\big| \right] = \lVert Z \rVert_{\infty,1}\,.
    \end{equation*}
    On the other hand,
    \begin{align*}
        \lVert Z \rVert_{\infty,1} \,&=\, \mathbb{E}_m\left[ \sup_{t\leq T} \big| Z_t \big| \right] = \mathbb{E}_m\left[ \sup_{t\leq T} \Big\{ e^{Mt}e^{-Mt}\big|Z_t\big| \Big\} \right] \,\leq&&\\
        &\leq\, \left(\sup_{t\leq T}e^{Mt}\right)\mathbb{E}_m\left[ \sup_{t\leq T} \Big\{ e^{-Mt}\big|Z_t\big| \Big\} \right]=e^{MT}\,\lVert Z \rVert_{\infty,1}^M.
    \end{align*}
    Therefore,
     the norms $\lVert \cdot \rVert_{\infty,1}$ and $\lVert \cdot \rVert_{\infty,1}^M$ are equivalent. Since one can easily prove that $\lVert \cdot \rVert_{\infty,1}^M$ is a norm,  as $\left(C_1,\lVert \cdot \rVert_{\infty,1}\right)$ is a Banach space, the last statement follows.
\end{proof}
We introduce a second norm that takes into account not only the discount term but also the time integral operator responsible for the non-Markovianity property. In particular, this is the main point distinguishing our well-posedness argument for the regularised Feynman--Kac-type equation with respect to a generic measure \eqref{eq:Feynman_Kac_generic_measure} from that of \cite{2016_Russo}.

\begin{lemma}\label{lemma_equiv_norme_Natalini}
    For any $M\geq 0$, $m\in\mathcal{P}\left(C\right)$, and $Z\in C_1$, we define
    \begin{equation*}
        \lVert Z \rVert_{\infty,1,1}^M \,:=\, \mathbb{E}_m\left[ \sup_{t\leq T} \Bigg\{ e^{-Mt} \int_0^t |Z_s|\,ds \Bigg\} \right] = \int_{C} \sup_{t\leq T} \Bigg\{ e^{-Mt} \int_0^t |Z_s(\omega)|\,ds \Bigg\} m(d\omega).
    \end{equation*}
     Then, $\lVert \cdot \rVert_{\infty,1,1}^M$ is a norm and it holds that there exists a positive constant $C_T$, depending only on $T$ and $M$, such that
    \begin{equation}\label{chain_of_inequalities}
        \lVert Z \rVert_{\infty,1,1}^M \,\leq\, T\,\lVert Z \rVert_{\infty,1} \,\leq\, C_T\,\lVert Z \rVert_{\infty,1}^M\,.
    \end{equation} 
\end{lemma}
\begin{proof}
    Since it is easy to show that $\lVert \cdot \rVert_{\infty,1,1}^M$ is a norm, we only prove the inequalities \eqref{chain_of_inequalities}. On the one hand,
    \begin{align*}
        \lVert Z \rVert_{\infty,1,1}^M \,&=\, \mathbb{E}_m\left[ \sup_{t\leq T} \Bigg\{ e^{-Mt} \int_0^t |Z_s|\,ds \Bigg\} \right]\,\leq\,\mathbb{E}_m\left[ \sup_{t\leq T} \Big\{ e^{-Mt}\cdot t\cdot\sup_{s\leq t} |Z_s| \Big\} \right]&&\\
        &=\, \sup_{t\leq T}e^{-Mt}\,\cdot\,\sup_{t\leq T}t\,\cdot\,\mathbb{E}_m\left[ \sup_{t\leq T} \Big\{ \sup_{s\leq t} |Z_s| \Big\} \right]=T\,\cdot\,\mathbb{E}_m\left[ \sup_{t\leq T} |Z_t| \right]=\,T\,\cdot\,\lVert Z \rVert_{\infty,1}.
    \end{align*}
    On the other hand,
    \begin{align*}
        \lVert Z \rVert_{\infty,1} \,&=\, \mathbb{E}_m\left[ \sup_{t\leq T} \big| Z_t \big| \right] = \mathbb{E}_m\left[ \sup_{t\leq T} \Big\{ e^{Mt}e^{-Mt}\big|Z_t\big| \Big\} \right]&&\\
        &\leq\, \left(\sup_{t\leq T}e^{Mt}\right)\mathbb{E}_m\left[ \sup_{t\leq T} \Big\{ e^{-Mt}\big|Z_t\big| \Big\} \right]=\, e^{MT}\,\lVert Z \rVert^M_{\infty,1}.
    \end{align*}
    Therefore, we get \eqref{chain_of_inequalities} with $C_T:=Te^{MT}$.
\end{proof}

We can finally state and prove the existence and uniqueness result for a solution to equation \eqref{eq:Feynman_Kac_generic_measure}.
\begin{theorem}
For any $m\in\mathcal{P}\left(C\right)$ and its  associated canonical process  $X^m$,  equation \eqref{eq:Feynman_Kac_generic_measure} admits a unique solution $u^m_K \in C^+$.
\end{theorem}
\begin{proof}
Let us fix a measure $m\in \mathcal{P}(C)$ and let $X:=X^m$ be the associated canonical process. The existence and uniqueness of the function $u^m_K$ is proved by means of a fixed point argument. Indeed, we may rewrite  \eqref{eq:Feynman_Kac_generic_measure} in the the Banach space $\left(C_1^+,\lVert \cdot \rVert_{\infty,1}^M\right)$ as 
\begin{equation}\label{eq:linking_fixed_point}
        u^m_K = \left(    T^m \circ\tau\right)\left(u^m_K\right).
    \end{equation}   
    The function $
        T^m : C_1^+ \rightarrow C\left( [0,T]\times\mathbb{R},\mathbb{R}^+ \right)$ maps $Z\in C_1^+$ to $T^m(Z)$ such that, for any $(t,y)\in [0,T]\times \mathbb R$,
    \begin{equation*}
        T^m(Z)(t,y) \,:=\, \int_{C} K\left( y - X_t(\omega) \right)V_t\left(\cdot, Z_\cdot(\omega) \right)m(d\omega),
    \end{equation*}
    and $ \tau : C\left( [0,T]\times\mathbb{R},\mathbb{R}^+ \right) \rightarrow C_1^+$  is
    such that, for any $f\in C\left( [0,T]\times\mathbb{R},\mathbb{R}^+ \right)$, $\tau(f)$ is defined as
    \begin{equation*}
         [\tau(f)]_t(\omega) \,:=\, f(t,\omega_t)
    \end{equation*}
    for every $t\in[0,T]$.
    
    We now show that $F=\tau \circ T^m$ admits a unique fixed point $Z\in C_1^+$. Given $Z,Z^\prime\in C_1^+$ and $(t,y)\in[0,T]\times\mathbb{R}$, by Lemma \ref{Lemma:V_continuity} and Definition \ref{def:kernel},
    \begin{align*}
        &|T^m(Z^\prime) - T^m(Z)|
       = \Big| \int_{C} K\left(y - X_t(\omega)\right) \left( V_t\left(\cdot, Z_\cdot^\prime(\omega) \right) - V_t\left( \cdot,Z_\cdot(\omega) \right) \right)\,m(d\omega) \Big|&&\\
        &\leq\, M_K\,\lambda^2\,c_0 \int_{C} \int_0^t\int_0^s |Z_{r}'(\omega) - Z_{r}(\omega)|\,dr\,ds\,m(d\omega)&&\\
        &\leq\, M_K\,\lambda^2\,c_0\,\mathbb{E}_m \left[  \sup_{s\leq t} \Bigg\{ e^{-M s} \int_0^s |Z_{r}' - Z_{r}|\,dr \Bigg\} \left(\int_0^t e^{M s}\,ds\right) \right]&&\\
        &=\, M_K\,\lambda^2\,c_0 \left(\int_0^t e^{M s}\,ds\right)\,\mathbb{E}_m \left[  \sup_{s\leq t} \Bigg\{ e^{-M s} \int_0^s |Z_{r}' - Z_{r}|\,dr \Bigg\} \right]&&\\
        &=\, M_K\,\lambda^2\,c_0\, \frac{e^{Mt} - 1}{M} \lVert Z^\prime - Z \rVert_{\infty,1,1}^M.
    \end{align*}
   
    Therefore, for $F:=\tau\circ T^m$, we obtain
    \begin{eqnarray*}
         \sup_{t\leq T}\left[e^{-Mt}\big| F(Z^\prime)_t - F(Z)_t \big|\right] \,&=&  \sup_{t\leq T}e^{-Mt}\big| T^m(Z^\prime)(t,X_t) - T^m(Z)(t,X_t) \big| \,\\
        &\leq\,& M_K\,\lambda^2\,c_0 \lVert Z^\prime - Z \rVert_{\infty,1,1}^M\,\sup_{t\leq T}\left(e^{-Mt}\,\frac{e^{Mt} - 1}{M}\right)\,\\
        &=&\, M_K\,\lambda^2\,c_0\, \frac{1}{M} \lVert Z^\prime - Z \rVert_{\infty,1,1}^M.
    \end{eqnarray*}
      
    Taking the expectation, by \eqref{chain_of_inequalities}, for $M>M_K\,\lambda^2\,c_0\,C_T$,  we get
    \begin{eqnarray*}
       \lVert \left( \tau \circ T^m \right)(Z^\prime)_t - \left( \tau \circ T^m \right)(Z)_t \rVert_{\infty,1}^M \,&\leq& \frac{M_K\,\lambda^2\,c_0}{M} \lVert Z^\prime - Z \rVert_{\infty,1,1}^M \,\\&\leq\, &\frac{M_K\,\lambda^2\,c_0\,C_T}{M} \lVert Z^\prime - Z \rVert_{\infty,1}^M\,\\&\leq\, &  \lVert Z^\prime - Z \rVert_{\infty,1}^M.
    \end{eqnarray*}
    
    Hence, we have that $\tau \circ T^m$ is a contraction on  the Banach space $\left(C_1,\lVert \cdot \rVert_{\infty,1}^M\right)$. Applying the Banach fixed point theorem, $\tau \circ T^m$ admits a unique fixed point $Z\in C_1^+$. Since 
    $$T^m(Z)=T^m(\tau\circ T^m(Z))=T^m\circ \tau (T^m(Z)),$$
    it is easy to prove that $T^m(Z)$ is a solution of \eqref{eq:linking_fixed_point}. This proves the existence of a solution to \eqref{eq:Feynman_Kac_generic_measure}. As for its uniqueness, it follows as in \cite[Theorem 3.1]{2016_Russo}.
\end{proof}

To conclude the analysis of the Feynman--Kac-type equation \eqref{eq:Feynman_Kac_generic_measure}, some regularity properties of its solution are derived.

\subsection{Regularity, stability and non-anticipating properties}\label{sec:stability_um}
The following boundedness and continuity properties hold.
\begin{proposition}
    Let $u:=u^m_K$ be a solution to \eqref{eq:Feynman_Kac_generic_measure}. Then, for any $(t,y)\in  [0,T]\times\mathbb{R}$, $u$ is bounded,
    $$|u|\le M_K,$$
    and Lipschitz continuous with respect to the second variable,
    $$|u(t,y) - u(t,y')|\le L_K |y - y^\prime|,$$
    where $M_K$ and $L_K$ are as in \eqref{eq:K_MK} and \eqref{eq:K_LK}, respectively.
\end{proposition}
\begin{proof}
   Boundedness follows directly from Lemma \ref{Lemma:V_continuity} and  \eqref{eq:K_MK}. Lipschitz continuity follows from Lemma \ref{Lemma:V_continuity}, \eqref{eq:K_LK}, and the fact that $m$ is a probability measure.  Indeed,  for every $t\in[0,T]$ and $y,y'\in\mathbb{R}$, we have that
    \begin{eqnarray*}
        |u(t,y) - u(t,y')| 
        &\leq&\,\int_{C } |K(y - X_t(\omega)) - K(y' - X_t(\omega))|\,|V_t\left( u(\cdot,X_\cdot(\omega)) \right)|m(d\omega) \\
        &\leq\,& \int_{C } |K(y - X_t(\omega)) - K(y' - X_t(\omega))|m(d\omega)\leq\,\,L_K\,|y - y'|.
    \end{eqnarray*}
\end{proof}

The next results regard some properties of continuity of $u^m_K$ with respect to the measure $m$ by considering the metric given by the Wasserstein distance between measures. For the convenience of the reader, we recall some basic properties of the Wasserstein distance.

Let us denote with $\mathcal{P}^2\left(C \right)$ the set of probability measures on $C $ with finite second moment.

\begin{definition}\label{def:Wasserstein}
    For any $\mu,\nu\in\mathcal{P}^2\left(C \right)$  the \emph{Wasserstein distance} $D_T$ between $\mu$ and $\nu$  is defined as
    \begin{equation*}
        D_T(\mu,\nu) \,:=\, \left[\inf_{\pi\in\Pi(\mu,\nu)} \int_{C \times C } \sup_{0\leq s\leq T} \lVert x_s - y_s\rVert^2 \pi\left(dx,dy\right)\right]^{1/2}\,,
    \end{equation*}
    where $\lVert\cdot\rVert$ denotes the Euclidean norm on $\mathbb{R} $, $x_s$ the projection at time $s\in[0,T]$ of the path $x\in C $ (similarly $y_s$), and $\Pi(\mu,\nu)$ is the space of coupling of $\mu$ and $\nu$. Moreover, for every $t\in[0,T]$, we also define
    \begin{equation*}
           D_t(\mu,\nu) \,:=\, \left[\inf_{\pi\in\Pi(\mu,\nu)} \int_{C \times C } \sup_{0\leq s\leq t} \lVert x_s - y_s\rVert^2 \pi\left(dx,dy\right)\right]^{1/2}.
        \end{equation*}
\end{definition}

\begin{lemma}\label{lemma:Wasserstein_properties}
    Let $\mu,\nu\in \mathcal{P}^2\left(C \right)$. For any $t\leq T$
    \begin{equation*}
        D_t(\mu,\nu) \,\leq\, D_T(\mu,\nu).
    \end{equation*}
    Also, let $Y^{\mu}$ and $Y^{\nu}$ be two stochastic processes such that $Law\left(Y^{\mu}\right)=\mu$ and $Law\left(Y^{\nu}\right)=\nu$. Then,
    \begin{equation}\label{eq:prop_Wasserstein}
        D_t(\mu,\nu)^2 \leq \mathbb{E}\left[ \sup_{0\leq s\leq t} \big\lVert Y_s^{\mu} \,-\, Y_s^{\nu} \big\rVert^2 \right]\,.
    \end{equation}
\end{lemma}
\begin{proof}
    For a complete proof of the results, see e.g. Chapter 6 in \cite{Villani}.
\end{proof}

In the following propositions, we state that for any $t\in [0,T]$, the solution $u^m_K(t,y)$ to equation \eqref{eq:Feynman_Kac_generic_measure} is continuous with respect to both the space variable and the measure.
\begin{proposition}\label{prop:stability_Feynman_Kac}
    Fix two measures $m,m'\in\mathcal{P}^2\left(C \right)$, and let $u^m_K,u^{m^\prime}_K$ be the respective solutions to  \eqref{eq:Feynman_Kac_generic_measure}. For any $t\in[0,T]$ and $y,y'\in\mathbb{R}$, we have
    \begin{equation}\label{eq:stability1}
        \big| u^m_K(t,y) - u^{m^\prime}_K(t,y') \big|^2 \,\leq\, C_1(t) \left(\big|y - y'\big|^2+ D_t(m,m')^2\right),
    \end{equation}     
    where $C_1(t):=8C^2Te^{2CTt}(1+t)$ and $C:=2L_K^2 + 2(\lambda^2c_0M_K)^2$. Furthermore, the function $(m,t,y)\mapsto u^m_K(t,y)$ is continuous on $\mathcal{P}\left(C \right)\times [0,T]\times \mathbb{R}$, where $\mathcal{P}(C )$ is endowed with the topology of weak convergence.
\end{proposition}
\begin{proof}See Appendix \ref{append_prop_eq:stability_1}.
\end{proof}

\begin{proposition}\label{eq:stability_2}
   For any $m,m'\in\mathcal{P}^2(C )$ and $t\in[0,T]$,
    \begin{equation*}
        \lVert u^m_K(t,\cdot) - u^{m^\prime}_K(t,\cdot) \rVert_2^2 \,\leq\, C_2\Big( 1+C_1(t) \Big)D_t(m,m')^2.
    \end{equation*}
    where $C_1(t)$ is as in Proposition \ref{prop:stability_Feynman_Kac} and $C_2:=2(M_K^\prime)^2+2M_K(\lambda^2c_0)^2$, with $M_K,M_K^\prime$ as in Definition \ref{def:kernel}.
\end{proposition}
\begin{proof} See Appendix \ref{append_prop_eq:stability_2}.
\end{proof}

The space of probability measures over path-space with finite second moment $\mathcal{P}^2(C )$ can be endowed with the topology of weak convergence. Moreover, one may introduce the following homogeneous distance on $\mathcal{P}^2(C )$:
\begin{equation}\label{eq:homogeneous_measure_distance}
    d_2(\eta,m)=\left[\sup_{\varphi\in\mathcal{A}}\mathbb{E}_\eta\left[ | \langle \eta - m,\varphi \rangle |^2 \right]\right]^{1/2},
\end{equation} where $\mathcal{A} := \big\{ \varphi\in C_b\left(C \right) :\lVert \varphi \rVert_{\infty} \leq 1\big\}$ and $\eta:\left(\Omega,\mathcal{F}\right) \rightarrow \left(\mathcal{P}^2(C),\mathcal{B}\mathcal{P}^2(C)\right)$ is a random measure. We have the following result.
\begin{proposition}\label{prop:stability_3}
    Suppose that $\mathcal{F}(K)\in L^1\left(\mathbb{R}\right)$, where $\mathcal{F}(K)$ denotes the Fourier transform of $K$. Then, for any random measure $\eta:\left(\Omega,\mathcal{F}\right) \rightarrow \left(\mathcal{P}^2(C),\mathcal{B}\mathcal{P}^2(C)\right)$, for all $t\in[0,T]$, and $m\in\mathcal{P}^2\left(C \right)$, we have that
    \begin{equation}\label{eq:stability_3}
        \mathbb{E}_\eta\left[ \lVert u^\eta_K(t,\cdot) - u^m_K(t,\cdot) \rVert_{\infty}^2 \right] \,\leq\, C_3(t)\,d_2(\eta,m)^2,
    \end{equation}
     where the distance $d_2(\eta,m)$ is defined by \eqref{eq:homogeneous_measure_distance}, $C_3(t):=C_{3,1}Te^{C_{3,2}(t)\,T\,t}$, $C_{3,1} := 1/{\sqrt{2\pi}} F_K^2$, $ C_{3,2}(t) := 1/{\sqrt{2\pi}}\,t^2\left(\lambda^2c_0\right)^2F_K^2$, and $F_K$ is as in \eqref{eq:K_FK}.
\end{proposition}
\begin{proof}
   The proof follows with the same arguments as in \cite[Lemma 3.4]{2016_Russo} and similar computations to the ones in Proposition \ref{prop:stability_Feynman_Kac} and Proposition \ref{eq:stability_2}.
\end{proof}

Moreover, the following non-anticipating property holds for $u^m_K$ solution to the generalised equation \eqref{eq:Feynman_Kac_generic_measure}.
\begin{definition}
    Fixed $t\in[0,T]$, let $m$ be a non-negative Borel measure on $\left(C,\mathcal{B}C\right)$, where $C = C \left([0,T],\mathbb{R}\right)$. $m_{|_{[0,t]}}$ will denote the unique induced measure on $\left(C _t,\mathcal{B}C _t\right)$, where $C _t:=C \left([0,t],\mathbb{R}\right)$, by
    \begin{equation*}
        \int_{C _t}F(\varphi)m_{|_{[0,t]}}(d\varphi) = \int_{C }F(\varphi_{|_{[0,t]}})\,m(d\varphi),
    \end{equation*}
    where $F:C _t\rightarrow \mathbb{R}$ is bounded and continuous.
\end{definition}

\begin{proposition}\label{prop:non_anticipating}
    For every $s\in[0,t]$ and $y\in\mathbb{R}$, it holds that $u^m_K(s,y)=u_K^{m_{|_{[0,t]}}}(s,y)$.
\end{proposition}
\begin{proof}
  The statement is a straight consequence of the existence of a unique solution of \eqref{eq:Feynman_Kac_generic_measure}.  
\end{proof}

Now, since the drift term in SDE \eqref{eq:SDE_u_K} explicitly involves the gradient of the solution to the Feynman--Kac-type equation \eqref{eq:mollified_equation}, it is necessary to derive regularity estimates also for $\nabla u^m_K$. This additional step was not required in \cite{2016_Russo}. To this end, we first establish the following auxiliary lemma and then derive regularity properties for $\nabla u^m_K$, which will be useful in the next results.
\begin{lemma}\label{Lemma:exp(u^m_K)_bounded}
    Given $m \in \mathcal{P}(C )$, let $u^m_K$ be the solution of \eqref{eq:Feynman_Kac_generic_measure}. Then there exist constants $C_m^\prime,C_M^\prime, C_m^{\prime\prime}\in (0,1]$ such that, for any $y\in\mathbb{R}$, we have
    \begin{align}
        &C_m^\prime\leq\exp\left( -\lambda   u^m_K\left( \cdot,y \right)(s) \right)\leq C_M^\prime;\label{Lemma:exp(u^m_K)_bounded_1}&&\\
        &C_m^{\prime\prime}\leq\exp\left(-\lambda\,c_0 \int_0^t \exp\left( -\lambda u^m_K\left( \cdot,y \right)(s)  \right) ds\right)\leq e^{-\lambda c_0 C_m^\prime t};\label{Lemma:exp(u^m_K)_bounded_2}&&\\
        &\big| \nabla u^m_K(t,y) \big| \leq {M_K^\prime}e^{-\lambda c_0 C_m' t}\,\leq\, M_K^\prime,&&\label{Lemma:exp(u^m_K)_bounded_3}
    \end{align} where $M^\prime_K$ is given in Definition \ref{def:kernel}.
\end{lemma}
\begin{proof}
All the inequalities are due to the fact that the map $t\mapsto \int_0^t u^m_K(s,y)\,ds$ is continuous over $[0,T]$ (see Proposition \ref{prop:stability_Feynman_Kac}), hence it is bounded in $[0,T]$.
\end{proof}

\begin{proposition}\label{prop:stability_Feynman_Kac:gradient}
    Fixed two measures $m,m'\in\mathcal{P}^2\left(C \right)$, let $u^m_K,u^{m^\prime}_K$ be the respective solutions of  \eqref{eq:Feynman_Kac_generic_measure}. For any $t\in[0,T]$ and $y,y'\in\mathbb{R}$, we have
    \begin{equation}\label{eq:stability:gradient}
        \big| \nabla u^m_K(t,y) - \nabla u^{m^\prime}_K(t,y') \big|^2 \,\leq\, \widetilde{C}_1(t) \left(\big|y - y'\big|^2+ D_t(m,m')^2\right),
    \end{equation}     
    where $t\mapsto \widetilde{C}_1(t)$ is an increasing function.
\end{proposition}
\begin{proof}
    Let us notice that, for any $m\in\mathcal{P}^2\left(C \right)$, given $u^m_K$ solution to \eqref{eq:Feynman_Kac_generic_measure},
    \begin{equation*}
        \nabla u^m_K(t,y) = \int_{C } \nabla K(y - X_t^m(\omega))\,\exp\left(  \int_0^t\Lambda\Big(s,u^m_K\left(\cdot,X^m_s(\omega)\right)(s)\Big) ds \right)m(d\omega).
    \end{equation*}
    Since by assumption \eqref{eq:K_MK} $\nabla K$ is bounded, the same arguments used in Proposition \ref{prop:stability_Feynman_Kac} can be used to show \eqref{eq:stability:gradient}.
\end{proof}

\begin{proposition}\label{prop:stability_3:gradient}
    Suppose that $\mathcal{F}(K)\in L^1\left(\mathbb{R}\right)$, where $\mathcal{F}(K)$ denotes the Fourier transform of $K$. Then, for any random measure $\eta:\left( \Omega,\mathcal{F} \right) \rightarrow \left(\mathcal{P}^2(C),\mathcal{B}\mathcal{P}^2(C)\right)$, for all $t\in[0,T]$, and $m\in\mathcal{P}^2(C)$, we have that
    \begin{equation}\label{eq:stability_3:gradient}
        \mathbb{E}_\eta\left[ \lVert \nabla u^\eta_K(t,\cdot) - \nabla u^m_K(t,\cdot) \rVert_{\infty}^2 \right] \,\leq\, \widetilde{C}_3(t)\,d_2(\eta,m)^2,
    \end{equation}
    where $d_2$ is again the homogeneous distance given by \eqref{eq:homogeneous_measure_distance} and $t\mapsto \widetilde{C}_3(t)$ is an increasing function.
\end{proposition}
\begin{proof}
    The same argument as in Proposition \ref{prop:stability_Feynman_Kac:gradient} allows to establish the result.
\end{proof}

Finally, we establish continuity and boundedness properties of the drift coefficient in SDE \eqref{eq:SDE_u_K} when evaluated at $u^m_K$. These properties will play a crucial role in proving the well-posedness of the McKean-Feynman-Kac SDE \eqref{eq:SDE_u_K},\eqref{eq:mollified_equation}. For every $m\in \mathcal{P}(C )$, it is defined as
\begin{equation}\label{eq:drift_b_integral_form} 
    b\big(u^m_K(\cdot,y)(t),\nabla u^m_K(\cdot,y)(t)\big):=\,-\varphi_1\lambda c_0\frac{\exp\left( -\lambda u^m_K(\cdot,y)(t) \right)\nabla u^m_K(\cdot,y)(t)}{\varphi_0+\varphi_1c_0 \exp\left( -\lambda \ u^m_K(\cdot,y)(t) \right)},   
\end{equation}
where the dependence on  the integral function $u^m_K(\cdot,y)(t):=\int_0^t u^m_K(s,y) ds$ has been made explicit.
Continuity and boundedness of the above drift are proven using the results in Section \ref{sect:well_posedness_reg_Feynman-Kac}.

\begin{lemma}\label{lemma:continuity_b} 
    Given $m,m'\in\mathcal{P}\left(C \right)$, let $u^m_K,u^{m^\prime}_K$ be the corresponding solutions to the equation \eqref{eq:Feynman_Kac_generic_measure}. Then the drift \eqref{eq:drift_b_integral_form} is Lipschitz continuous with respect to the second  and third variables, i.e.  
    \begin{equation}\label{eq:lipschitz_b}
        \begin{aligned}
            &\Big| b\Big(u^m_K(\cdot,y)(t),\nabla u^m_K(\cdot,y)(t)\Big) - b\Big(u^{m^\prime}_K(\cdot,y)(t),\nabla u^{m^\prime}_K(\cdot,y)(t)\Big) \Big|\\
            &\leq\widetilde{C}_1 \left(\Big|u^m_K(\cdot,y)(t) - u^{m^\prime}_K(\cdot,y)(t)\Big| +  \Big|\nabla u^m_K(\cdot,y)(t) - \nabla u^{m^\prime}_K(\cdot,y)(t)\Big|\right),
        \end{aligned}
    \end{equation}
    where $ \widetilde{C}_1= \widetilde{C}_1(  \varphi_0, \varphi_1,\lambda,c_0, M_K^\prime).$
\end{lemma}
\begin{proof}
    See Appendix \ref{section:continuity_b}.
\end{proof}

\begin{lemma}\label{lemma:boundedness_drfit}
    Let $m\in C$. The drift \eqref{eq:drift_b_integral_form} is bounded, i.e.  for every $t\in[0,T]$ and $y\in\mathbb{R}$,
    \begin{equation}\label{eq:boundedness_b}
        \big|b(u^m_K(\cdot,y)(t),\nabla u^m_K(\cdot,y)(t)) \big|\le \widetilde{C}_2,
    \end{equation}
where $ \widetilde{C}_2= \widetilde{C}(  \varphi_0, \varphi_1,\lambda,c_0, M_K^\prime). $
\end{lemma}
\begin{proof}
    Note that
    \begin{equation*}
          \big|b\big(u^m_K(\cdot,y)(t),\nabla u^m_K(\cdot,y)(t)\big)  \big| = \varphi_1\lambda c_0\frac{\exp\left( -\lambda u^m_K(\cdot,y)(t) \right)\left|\nabla u^m_K(\cdot,y)(t)\right|}{\varphi_0+\varphi_1c_0 \exp\left( -\lambda \ u^m_K(\cdot,y)(t) \right)}.
    \end{equation*}
    We observe that if $\varphi_1\ge 0$ then $\varphi_0+\varphi_1c_0 \exp\left( -\lambda \ u^m_K(\cdot,y)(t) \right) \ge \varphi_0$; otherwise, it holds that $\varphi_0+\varphi_1c_0 \exp\left( -\lambda \ u^m_K(\cdot,y)(t) \right) \ge \varphi_0+\varphi_1c_0.$  Then, the thesis derives from  Lemma \ref{Lemma:exp(u^m_K)_bounded}   with $\widetilde{C}_2=  {\varphi_1 c_0\lambda M_K'}/{\varphi_0}$ if $\varphi_1\ge 0$; otherwise, $\widetilde{C}_2= {\varphi_1 c_0\lambda M_K'}/{\varphi_0+\varphi_1 c_0}$.
\end{proof}

From the previous results, existence and uniqueness of the solution to equation \eqref{eq:mollified_equation} are guaranteed, as well as the above regularity results. We now turn to the well-posedness of SDE \eqref{eq:SDE_u_K}.

\section{Strong existence and uniqueness of the MKFK-SDE}\label{sec:Existence_uniqueness_SDE}
Let us recall the McKean-type SDE
\begin{equation}\label{eq:SDE_u_K_proof_well_posed}
        Y_t = Y_0 + \int_0^t b\left(u_K(\cdot,Y_s)(s),\nabla  u_K(\cdot,Y_s)(s)\right)ds + \sqrt{2}W_t,\quad t\in[0,T],
\end{equation}
with the drift given by the velocity field \eqref{eq:def_b_PDE}, coupled with the McKean Feynman--Kac-type equation
\begin{equation}\label{eq:mollified_equation_proof_well_posed}
    u_K(t,y)= \mathbb E\left[K(y-Y_t) \exp\left(-\lambda\,c_0 \int_0^t \exp\left( -\lambda u_K(\cdot,Y_s)(s)  \right) ds\right) \right],
\end{equation}

The following well-posedness result holds.
\begin{theorem}\label{thm:existence_uniqueness_SDE}
    The McKean-Feynman-Kac SDE \eqref{eq:SDE_u_K_proof_well_posed}-\eqref{eq:mollified_equation_proof_well_posed} admits a pathwise unique strong solution. Also, \eqref{eq:SDE_u_K_proof_well_posed}-\eqref{eq:mollified_equation_proof_well_posed} admits a weak solution, which is unique in the sense of probability law.
\end{theorem}
\begin{proof}
    Fix $m,m^\prime\in\mathcal{P}^2(C)$, and let $Y^m$ solution to
    \begin{equation}\label{eq:SDE_fixed_measure}
        Y^m_t = Y^m_0 + \int_0^t b\left(u^m_K(\cdot,Y^m_s)(s),\nabla  u^m_K(\cdot,Y^m_s)(s)\right)ds + \sqrt{2}W_t,\quad t\in[0,T],
    \end{equation}
    with $Y^m_0\sim\rho_0(x)dx$, where $u^m_K(\cdot,Y^m_s)(s):=\int_0^su^m_K(r,Y^m_s)dr$, and $u^m_K(r,x)$ solution to the generalised Feynman--Kac-type equation
    \begin{equation}\label{eq:Feynman_Kac_fixed_measure}
        u^m_K(r,x) = \int_{C} K(y - X_r^m(\omega))\,\exp\left(  \int_0^r-\lambda c_0\exp\left( -\lambda \int_0^v u^m_K\left(s,X^m_v(\omega)ds\right) \right)dv\right)m(d\omega).
    \end{equation}
    $Y^{m^\prime}$ is defined mutatis mutandis.

    We start by proving that
    \begin{equation}\label{eq:diff_Y_andwasserstein_ex_uniq_MKV_SDE}
        \mathbb{E}\left[ \sup_{0\leq r\leq t} \big| Y_r^{m^\prime} - Y^m_r \big|^2 \right] \,\leq\, 4C(T)\int_0^t D_s(m,m')^2\,ds\,.
    \end{equation}
    
    In particular, once we have proved that the following bound
    \begin{equation}\label{eq:diff_Y}
        \begin{aligned}
            &\mathbb{E}\left[ \sup_{0\leq r\leq t} \big| Y_r^{m^\prime} - Y^m_r \big|^2 \right]\\
            &\leq 2C(T)\int_0^t\int_0^s \mathbb{E}\left[ \Big| u^{m^\prime}_K(r,Y^{m'}_s) -u^m_K(r,Y^m_s) \Big|^2 + \Big| \nabla u^{m^\prime}_K(r,Y^{m'}_s) - \nabla u^m_K(r,Y^m_s) \Big|^2 \right]drds
        \end{aligned}
    \end{equation}
    holds, inequality \eqref{eq:diff_Y_andwasserstein_ex_uniq_MKV_SDE} derives from the results in Subsection \ref{sec:stability_um}. Indeed,
    \begin{align}
        &\mathbb{E}\left[ \sup_{0\leq \ell\leq t} \big| Y_{\ell}^{m^\prime} - Y^m_{\ell} \big|^2 \right]\notag\\
        &\begin{aligned}
            \,\,\leq\,\mathbb{E}\Bigg[ \sup_{0\leq \ell\leq t} \ell \int_0^{\ell} \Big|&b\left(u^m_K(\cdot,Y^m_s)(s),\nabla u^m_K(\cdot,Y^m_s)(s)\right)\\
            &- b(u^{m^\prime}_K(\cdot,Y^{m'}_s)(s),\nabla u^{m^\prime}_K(\cdot,Y^{m'}_s)(s)) \Big|^2ds \Bigg]
        \end{aligned}\label{eq:proof_uniq_ineq_1}\\
        &\begin{aligned}
            \,\,=\,\mathbb{E}\Bigg[ t \int_0^t \Big|& b\left(u^m_K(\cdot,Y^m_s)(s),\nabla u^m_K(\cdot,Y^m_s)(s)\right)\\
            &- b(u^{m^\prime}_K(\cdot,Y^{m'}_s)(s),\nabla u^{m^\prime}_K(\cdot,Y^{m'}_s)(s))\Big|^2ds \Bigg]
        \end{aligned}\notag\\
        &\begin{aligned}
            \,\,\leq\, 2C(T)\int_0^t ds\int_0^s  dr \Bigg(& \mathbb{E}\left[ \Big| u^{m^\prime}_K(r,Y^{m'}_s) - u^m_K(r,Y^m_s) \Big|^2 \right]\\
            &+ \mathbb{E}\left[ \Big| \nabla u^{m^\prime}_K(r,Y^{m'}_s) - \nabla u^m_K(r,Y^m_s) \Big|^2 \right]  \Bigg)
        \end{aligned}\label{eq:proof_uniq_ineq_2}\\
        &\leq 4C(T)\int_0^t\left(\mathbb{E}\left[\sup_{r\leq s}\big|Y_r^{m'}-Y_r^m\big|^2\right] + \int_0^s D_r(m,m')^2dr\right)ds,\label{eq:proof_uniq_ineq_3}
    \end{align}
     where $C(T):=T\left(C_1(T)\vee\widetilde{C}_1(T)\right)$. \eqref{eq:proof_uniq_ineq_1} is due to Cauchy-Schwartz's inequality; \eqref{eq:proof_uniq_ineq_2} by Lemma \ref{lemma:continuity_b}; \eqref{eq:proof_uniq_ineq_3} by Proposition \ref{prop:stability_Feynman_Kac} and Proposition \ref{prop:stability_Feynman_Kac:gradient}.
     
     By Gronwall's lemma, by taking the supremum inside the double integral on the right-hand side, and the first result in Lemma \ref{lemma:Wasserstein_properties},
     \begin{equation*}
         \mathbb{E}\left[ \sup_{0\leq \ell\leq t} \big| Y_{\ell}^{m^\prime} - Y^m_{\ell} \big|^2 \right]\leq 4C(T)\int_0^t D_s(m,m^\prime)^2ds
     \end{equation*}
          By Lemma \ref{lemma:Wasserstein_properties}  and \eqref{eq:diff_Y_andwasserstein_ex_uniq_MKV_SDE},
    \begin{equation}\label{eq:D_Esu_D_proof_ex_uniq_MKV_SDE}
        D^2_t(m,m^\prime)  \leq \mathbb{E}\left[ \sup_{0\leq s\leq t} \big| Y_s^{m^\prime} - Y^m_s \big|^2 \right]\, \leq 4C(T)\,\int_0^t D_r(m,m^\prime)^2\,dr\,.
    \end{equation}
    Hence, again by Gronwall's lemma, $D^2_T(m,m^\prime) = 0$. As a consequence, 
    from \eqref{eq:D_Esu_D_proof_ex_uniq_MKV_SDE}, we get  $$\mathbb{E}\left[ \sup_{0\leq s\leq t} \Big| Y_s^{m^\prime} - Y^m_s \Big|^2 \right] = 0,$$ which implies the pathwise uniqueness of a solution to equation \eqref{eq:SDE_u_K_proof_well_posed}.
    
    Let us turn to the existence of a weak solution to SDE \eqref{eq:SDE_u_K_proof_well_posed}. As $\left(\mathcal{P}^2\left(C \right),D_T\right)$ is a complete separable metric space, we can show the existence of such a solution by the fixed point argument. Consider the map $\Theta:\mathcal{P}^2\left(C \right)\,\rightarrow\,\mathcal{P}^2\left(C \right)$ such that for any measure $\mu \in \mathcal{P}^2\left(C \right)$,
    \begin{equation}\label{eq:def_theta}
        \Theta(\mu)= \mathcal{L}\left(Y^{\mu}\right),
    \end{equation}
    where $Y^{\mu}$ is a solution to \eqref{eq:SDE_fixed_measure}-\eqref{eq:Feynman_Kac_fixed_measure} with $m=\mu$. For a fixed measure $\mu\in\mathcal{P}^2(C)$, the existence of $Y^{\mu}$ is granted by the Lipschitz continuity \eqref{eq:lipschitz_b} and boundedness \eqref{eq:boundedness_b} of the drift. Moreover, it is straightforward to see that $\Theta$ is well-defined on $\mathcal{P}^2(C)$; indeed, since the drift is bounded (see Lemma \ref{lemma:boundedness_drfit}) and Brownian motion has finite moments of all orders, we obtain
    \begin{equation*}
        \sup_{0\leq t\leq T}\mathbb{E}\left[\rVert Y^\mu_t\rVert^2\right] <\infty.
    \end{equation*}
    Now, the goal is to prove that $\Theta$ is a contraction with respect to the Wasserstein metrics, i.e.  $\Theta(\mu)=\mu$.

    Given $\mu, \nu \in \mathcal{P}^2\left(C \right)$, for any $t\in[0,T]$,
    \begin{equation}\label{eq:D2Theta_D2_proof_ex_uniq_MKV_SDE}
        D_t\left(\Theta(\mu),\Theta(\nu)\right)^2  \,\leq\, \mathbb{E}\left[ \sup_{0\leq s\leq t} \big| Y_s^{\mu} - Y^{\nu}_s \big|^2 \right] \,\leq\,  4C(T)\,\int_0^t D_r(\mu,\nu)^2\,dr\,,
    \end{equation}
    where the first inequality is due to \eqref{eq:D_Esu_D_proof_ex_uniq_MKV_SDE} and Lemma \ref{lemma:Wasserstein_properties}; while the second inequality comes from \eqref{eq:diff_Y_andwasserstein_ex_uniq_MKV_SDE}.
   
    Now, fix a measure $m\in \mathcal{P}^2\left(C \right)$.   By iteratively applying   the map $\Theta$  to the measure $m$, we define a sequence of measures $  \{m_k\}_{k\in \mathbb{N}}$ in $\mathcal{P}\left(C \right)$  such that
    \begin{equation}\label{eq:def_m_k_ex_uni_SDE}
        m_k\,:=\,  \Theta^k(m) = \Theta\left(\Theta^{k-1}(m)  \right).
    \end{equation} 
    
    The sequence $\{m_k\}_{k\in\mathbb{N}}$ is a Cauchy sequence in $\left(\mathcal{P}^2(C ), D_T)\right)$. Indeed, by an iteration of  (\ref{eq:D2Theta_D2_proof_ex_uniq_MKV_SDE}), one gets 
    \begin{equation*}
         D_T\left( \Theta^{k+1}(m),\Theta^k(m) \right)^2  \leq \frac{\left[4C(T)\right]^k}{k!}\,D_T\left( \Theta(m),m \right)^2.
    \end{equation*}

    Therefore, for any $\ell,n\in\mathbb{N}$,
    \begin{equation*}
         D_T\left( \Theta^\ell(m),\Theta^n(m) \right)^2  \leq D_T\left( \Theta(m),m \right)^2\sum_{k=\ell}^n\frac{\left[4C(T)\right]^k}{k!}.
    \end{equation*}
    Since the tail
    \begin{equation*}
        \sum_{k=\ell}^\infty\frac{\left[4C(T)\right]^k}{k!} \xrightarrow{\ell\rightarrow\infty} 0,
    \end{equation*}
    $\{m_k\}_{k\in\mathbb{N}}$ is a Cauchy sequence. Therefore, there exists a probability measure $\mathbb{Q}$ such that $m_k$ weakly converges to  $\mathbb{Q}$ as $k\rightarrow +\infty$ and $\Theta(\mathbb{Q}) = \mathbb{Q}$, by definition of $m_k$. By construction of $\Theta$, this implies that there exists a weak solution to SDE \eqref{eq:SDE_u_K_proof_well_posed}. By the results of Yamada and Watanabe \cite[308]{Karatzas}, there exists a pathwise unique strong solution to SDE  \eqref{eq:SDE_u_K_proof_well_posed}.

    As for the uniqueness in law, we follow the proof of the second claim in \cite[Theorem 3.1]{2016_Russo}, applied to our case. 
    Again, let $\left(Y,m\right)$ and $\left(Y',m'\right)$ be two solutions to equation \eqref{eq:SDE_u_K_proof_well_posed} on possibly different probability spaces, Brownian motions, and initial condition distributed according to $\zeta_0$ such that $\mathcal{L}(\zeta_0)=\mu_0$ and $\mu_0(dx)=\rho_0(x)dx$. Also, let $m:=\mathcal{L}(Y)$ and $m^\prime:=\mathcal{L}(Y^\prime)$.

    Given $\nu\in\mathcal{P}^2\left(C \right)$, we indicate by $\Theta(\nu)$ the law of $\overline{Y}$, where $\overline{Y}$ is the strong solution of
    \begin{equation}\label{eq:Y_bar}
        \overline{Y}_t = Y_0 + \int_0^t b\big( u^{m^\prime}_K(\cdot,\overline{Y}_s)(s),\nabla u^{m^\prime}_K(\cdot,\overline{Y}_s)(s)\big)ds + \sqrt{2}\,W_t
    \end{equation}
    on the same probability space and same Brownian motion driving $Y$. Therefore, by \eqref{eq:D_Esu_D_proof_ex_uniq_MKV_SDE}, $\mathcal{L}(\overline{Y})=m$.

    Since $m'$ is fixed, $\overline{Y}$ is solution of a classical SDE for which pathwise uniqueness holds. In fact, thanks to Proposition \ref{prop:stability_Feynman_Kac}, Proposition \ref{prop:stability_Feynman_Kac:gradient}, and Lemma \ref{lemma:continuity_b}, it is possible to apply Theorem 11.2 in \cite[128]{Rogers_Williams_2} which assures that there exists a pathwise unique solution $\overline{Y}$ to \eqref{eq:Y_bar}. By Yamada-Watanabe theorem, $Y'$ and $\overline{Y}$ have the same distribution. Consequently, $\Theta(m')=\mathcal{L}\left( \overline{Y} \right)=\mathcal{L}\left( Y' \right)=m'$. Since $\mathcal{L}(\overline{Y})=m$, we conclude.
\end{proof}

Having established the well-posedness of the limiting SDE, we now turn to the associated interacting particle system. Our goal is to study the convergence of the corresponding weighted empirical density towards the solution of the limiting Feynman--Kac-type equation, thereby establishing the mean-field limit.

\section{Particle system and propagation of chaos}\label{sec:propagation_chaos}
Finally, we give an interpretation of the solution to the PDE \eqref{eq:Natalini_density} as the limiting density of a particle system as the number of particles goes to infinity, and  we prove that $Y$ represents the evolution of the typical particle under the influence of the field $u_K$. From the mathematical point of view, this means that, at the microscale, we expect that any $N\in \mathbb N$ molecules positions among infinite others may be described as independent processes $Y^i=(Y^i_t)_{t\in [0,T]}$, for $i=1,\ldots,N$, with common law $m=\mathcal{L}(Y)$, where $Y$ is the solution to the SDE \eqref{eq:SDE_u_K}. Equivalently, given a family of $N$ independent Brownian motions $(W^1,\ldots,W^N)$ and a vector of $N$ independent and identically distributed random variables $(Y_0^1,\ldots, Y_0^N)$ such that $Y_0^i \sim Y_0 $ and $\mathcal{L}(Y_0)=\rho_0(\cdot)dx$, for any $i=1,\ldots,N$,  the process $ Y^i $ is the  solution of    the system
\begin{equation}\label{eq:system_particles_iid}
    \begin{split}
        Y_t^i &= Y_0^i + \int_0^t b\big(u_K\left(\cdot,Y_s^i\right)(s),\nabla u_K\left(\cdot,Y_s^i\right)(s)\big)ds + \sqrt{2}W_t^i;\\
        u_K(t,y) &= \mathbb{E}\left[ K\left(y - Y_t^i\right)\exp\left(-\lambda\,c_0 \int_0^t \exp\left( -\lambda     u^m_K(\cdot,Y^i_s)(s)  \right) ds\right) \right];\\
        m &:= \mathcal{L}\left( Y^i \right),
    \end{split}
\end{equation}
where the drift is explicitly given by \eqref{eq:drift_b_integral_form}.
System \eqref{eq:system_particles_iid} describes the typical behaviour of $N$ identical and independent molecules among infinite others, which follows a common field $u^m_K$, i.e.  the propagation of chaos property is achieved. Indeed,  let us introduce the following interacting particle system associated with \eqref{eq:system_particles_iid}
\begin{equation}\label{eq:system_particles_interacting}
\begin{split}
        \xi_t^{i,N} & = \xi_0^{i,N} + \int_0^t b\left(u^{\mu_N^\xi}\left( \cdot,\xi_s^{i,N} \right)(s),\nabla u^{\mu_N^\xi}\left( \cdot,\xi_s^{i,N} \right)(s)\right)ds + \sqrt{2}\,W_t^i\,;\\
         u^{\mu_N^\xi}(t,y) &=\frac{1}{N}\sum_{j=1}^N K\left( y - \xi_t^{i,N} \right) \exp\left(-\lambda c_0 \int_0^t \exp\left(-\lambda \int_0^s u^{\mu_N^\xi}\left( r,\xi^{i,N}_s \right) dr\right) ds\right),
         \end{split}
\end{equation}
with initial conditions $\xi_0^{i,N} = Y_0^i$. In \eqref{eq:system_particles_interacting}, $\mu_N^\xi \in \mathbb{P}(\mathcal{P}^2(C))$ is  the empirical measure of the $N$ molecules, i.e.  its marginal at time $t$, for any $t\in[0,T]$, is given by
\begin{equation}\label{eq:empirical_measure_mu_N_xi}
         \mu_N^\xi(t):= \frac{1}{N} \sum_{i=1}^N \varepsilon_{\xi^{i,N}_t}.
\end{equation} 
Let us stress that the particle system consists of interchangeable particles. Also, we refer to the second equation of \eqref{eq:system_particles_interacting} as the \emph{empirical Feynman--Kac-type equation} and to the system \eqref{eq:system_particles_interacting}  as the \emph{empirical MKFK system}.
\begin{proposition}\label{prop:estimate_d2}
    Let us define the empirical measure $ \mu_N^Y$ associated to the $N$ independent particles $(Y^1,\ldots,Y^N)$ solution of \eqref{eq:system_particles_iid} as
    \begin{equation} \label{eq:empirical_measure_Y}
        \mu_N^Y :=\frac{1}{N}\sum_{i=1}^N\varepsilon_{Y^i}\in\mathbb{P}(\mathcal{P}^2(C)).
    \end{equation}
    Then the following bound holds
    \begin{equation}\label{eq:estimate_d2}
        d_2(\mu_N^Y,m)^2 \leq\frac{1}{N},
    \end{equation}
    where the distance $d_2$ is defined in \eqref{eq:homogeneous_measure_distance}.
\end{proposition}
\begin{proof}
    For the independence of the processes $Y^i$, with $i=1,...,N$, by applying   \eqref{eq:homogeneous_measure_distance}, as $\|\varphi\|\le 1$, we obtain
    \begin{equation*}
        d_2(\mu_N^Y,m)^2= \sup_{\varphi\in\mathcal{A}} \mathbb{E}\left[ \Bigg( \frac{1}{N}\sum_{j=1}^N \varphi(Y^j) - \mathbb{E}\left[ \varphi(Y^j) \right] \Bigg)^2 \right]\,\leq\, \sup_{\varphi\in\mathcal{A}} \frac{1}{N}Var\left(\varphi(Y)\right)\,\leq\,\frac{1}{N}\,.
    \end{equation*}
\end{proof}
\begin{theorem}
    For any $N\in \mathbb N$, the solution $(\xi^{1,N},\ldots, \xi^{N,N})$  of the empirical MKFK system \eqref{eq:system_particles_interacting} admits a strong solution which is pathwise unique.
\end{theorem}
\begin{proof} 
    Let us start by denoting the space $C \left([0,T],\mathbb{R}^N\right)$ with $C ^N$. The well-posedness of system \eqref{eq:system_particles_interacting} is a consequence of  Theorem 11.2  in \cite[128]{Rogers_Williams_2} for  path-dependent SDE, once the measurability and Lipschitz continuity of the  map
    \begin{equation}\label{eq:prop_wellposed_PS}
        \begin{aligned}
            \left( [0,T]\,\times\, C^N, \mathcal{B}[0,T]\otimes \mathcal{B}C^N\right) &\longrightarrow \left(\mathbb{R}^N,\mathcal{B}\mathbb{R}^N \right)\\
            \left(t,\xi\right) &\mapsto \left( u^{\mu_N^\xi}\left(t,\xi_t^1\right),\ldots,  u^{\mu_N^\xi}\left(t,\xi_t^N\right)\right)
        \end{aligned}   
    \end{equation}
    are proved. We start with the measurability of \eqref{eq:prop_wellposed_PS}. $\mu_N^\xi $ is a measurable map from $\left(\Omega,\mathcal{F}\right)$ to $\left( \mathcal{P}\left(C \right),\mathcal{B}\mathcal{P}(C) \right)$ such that $\mu_N^{\xi}\in\mathcal{P}^2\left(C \right)$ $\mathbb{P}$-a.s.. Moreover, for any $i\in\{1,...,N\}$, the map
    \begin{equation*} 
        \begin{aligned}
            \left( [0,T]\times C^{N}, \mathcal{B}[0,T]\otimes \mathcal{B}C^N\right) &\longrightarrow \left( \mathcal{P}\left(C \right)\times[0,T]\times\mathbb{R} ,\mathcal{B}\mathcal{P}(C)\otimes \mathcal{B}[0,T]\otimes \mathcal{B}\mathbb{R}\right)\\
            \left(t,\bar{\xi}\right)&\mapsto \left( \mu_N^{\bar{\xi}},t,\bar{\xi}_t^i \right)
        \end{aligned}   
    \end{equation*}
    is also continuous. Since the map $(m,t,y)\mapsto u^m_K(t,y)$ is continuous (see Proposition \ref{prop:stability_Feynman_Kac}), the map \eqref{eq:prop_wellposed_PS}, being the composition of continuous functions, is also continuous, and therefore measurable. Furthermore, the drift is clearly adapted and then progressively measurable.

   Let $\xi,\eta\in C ^N$. By  Proposition \ref{prop:stability_Feynman_Kac}, for any $t\in[0,T]$ and for any $i\in\{1,...,N\}$, we get
    \begin{eqnarray*}
        \Big| u^{\mu_N^\xi} \left(t,\xi_t^i\right)  - u^{\mu_N^\eta} \left(t,\eta_t^i\right) \Big| 
        &\leq& C_1(t) \Big[ \big|\xi_t^i - \eta_t^i \big| + D_t\left(\mu_N^\xi,\mu_N^\eta \right) \Big]\\        &\leq&\, C_1(t)\Big[ \big|\xi_t^i - \eta_t^i \big| + \frac{1}{N}\sum_{j=1}^N \sup_{s\leq t}\big| \xi_s^j - \eta_s^j \big| \Big]\\
        &\leq&\,2 C_1(t)\max_{j=1,...,N}\sup_{s\leq t}\big| \xi_s^j - \eta_s^j \big|.
    \end{eqnarray*}
    With the same procedure, one may show the same properties for the $\nabla u^{\mu_N^\xi}$. Hence, we conclude by the uniform Lipschitz continuity of $b$ (see Lemma \ref{lemma:continuity_b}) and Theorem 11.2 in \cite[128]{Rogers_Williams_2}.
\end{proof}

Given the well-posedness of the particle system, we may prove the propagation of chaos property. Let us recall the definition of such a property.

\begin{definition}\label{def:chaos_propagation}
    For any given  $N\in \mathbb N$, let $(Y^1,\ldots,Y^N)$ and $(\xi^1,\ldots,\xi^N)$ be the solutions of the systems \eqref{eq:system_particles_iid} and \eqref{eq:system_particles_interacting},  respectively. \emph{Propagation of chaos} occurs whenever for any fixed $k\ge 2$,
    \begin{equation*}
        \left( \xi^1,\xi^2,...,\xi^k \right)\xrightarrow[\mathcal{L}]{N\to \infty} \left( Y^1,Y^2,...,Y^k \right),
    \end{equation*}
    or, equivalently,
    \begin{equation*}
        \mathcal{L}\left( \xi^1,\xi^2,...,\xi^k \right)\xrightarrow{N\to \infty}\otimes_k m,
    \end{equation*}
    that is the joint law of any $k$-th dimensional process solution of \eqref{eq:system_particles_interacting} converges to the $k$-dimensional product law of $k$ independent copies of the law $m$ of the solution of \eqref{eq:system_particles_iid}.
\end{definition} 

The propagation of chaos derives from the following proposition, which provides some a priori estimates with respect to the number of particles $N\in\mathbb{N}$.
\begin{proposition}\label{prop:boundness_interactin_particles}
    Let $N\in \mathbb N$. Let $(Y^1,\ldots,Y^N)$ be the solution to the system \eqref{eq:system_particles_iid} with $u^m_K$, and $(\xi^1,\ldots,\xi^N)$ be the  solution to the system  \eqref{eq:system_particles_interacting} with $u^{\mu_N^\xi}$ associated to the empirical measure $\mu_N^\xi$ given by \eqref{eq:empirical_measure_mu_N_xi}.  
   
    Then, there exists a finite   constant $C \in \mathbb R_+$ such that, for any  $t\in [0,T]$,
    \begin{equation}\label{eq:boundness_interactin_particles}
        \mathbb{E}\left[ \lVert u^{\mu_N^\xi}(t,\cdot) - u_K(t,\cdot) \rVert_{\infty}^2 \right] + \mathbb{E}\left[ \lVert \nabla u^{\mu_N^\xi}(t,\cdot) - \nabla u_K(t,\cdot) \rVert_{\infty}^2 \right] + \sup_{i=1,...,N}\mathbb{E}\left[ \sup_{s\leq t} \Big| \xi_s^i - Y_s^i \Big|^2 \right]\leq \frac{C}{N},
    \end{equation} 
    where $C=C\left( M_K,L_K,T,\lambda,c_0 \right)$. Furthermore, if $\mathcal{F}(K)\in W^{1,2}\left( \mathbb{R} \right)$, there exists a finite positive constant $C$ such that, for all $t\in [0,T]$,
    \begin{equation}\label{eq:boundness_interactin_particles2}
        \mathbb{E}\left[ \lVert u^{ \mu_N^\xi}(t,\cdot) - u_K(t,\cdot) \rVert_2^2 \right] \,\leq\, \frac{C}{N}\,,
    \end{equation}
    where $C=C\left( M_K,L_K,T,\lVert \nabla K\rVert_2,\lambda,c_0 \right)$.
\end{proposition}
\begin{remark}
    The convergence analysis must also include the gradient of the empirical Feynman--Kac density, since it appears explicitly in the drift coefficients of both the interacting particle system and the associated limiting SDE. This constitutes a further difference with respect to the framework considered in \cite{2016_Russo}, where convergence of the density alone was sufficient.
\end{remark}
\begin{proof}
    Let us recall that the map \eqref{eq:prop_wellposed_PS} is measurable and  $\left( u^{\mu_N^\xi}\left(t,\xi_t^1\right),\ldots,  u^{\mu_N^\xi}\left(t,\xi_t^N\right)\right)$ satisfies the non-anticipating property by Proposition \ref{prop:non_anticipating}. For any $i\in\{1,...,N\}$, by the same computations as in \eqref{eq:diff_Y}, Proposition \ref{prop:stability_Feynman_Kac}, and Proposition \ref{prop:stability_Feynman_Kac:gradient}, we get
    \begin{align*}
        &\mathbb{E}\left[ \sup_{0\leq s\leq t} \big| \xi_s^i - Y_s^i \big|^2 \right]\\
        &\leq2C(T)\mathbb{E}\left[ \int_0^t \int_0^s \left( \Big| u^{\mu_N^\xi}(r,\xi_s^i) - u_K(r,Y_s^i) \Big|^2 + \Big| \nabla u^{\mu_N^\xi}(r,\xi_s^i) - u_K(r,Y_s^i) \Big|^2 \right)drds \right]\\
        &\begin{aligned}
            \,\,\leq\,&4C(T)\mathbb{E}\left[ \int_0^t \int_0^s \left(\Big| u^{\mu_N^\xi}(r,\xi_s^i) - u_K(r,\xi_s^i) \Big|^2 + \Big| \nabla u^{\mu_N^\xi}(r,\xi_s^i) - \nabla u_K(r,\xi_s^i) \Big|^2\right)drds \right]\\
            &+4C(T)\mathbb{E}\left[ \int_0^t\int_0^s \left(\Big| u_K(r,\xi_s^i) -  u_K(r,Y_s^i) \Big|^2 + \Big| \nabla u_K(r,\xi_s^i) -  \nabla u_K(r,Y_s^i) \Big|^2\right)drds \right]
        \end{aligned}\\
        &\begin{aligned}
            \,\,\leq\,&4C(T)\int_0^t \int_0^s\mathbb{E}\left[\left(\Big| u^{\mu_N^\xi}(r,\xi_s^i) - u_K(r,\xi_s^i) \Big|^2 + \Big| \nabla u^{\mu_N^\xi}(r,\xi_s^i) - \nabla u_K(r,\xi_s^i) \Big|^2\right)\right]drds\\
            &+8TC(T)\int_0^t \sup_{i=1,\dots,N}\mathbb{E}\left[ \sup_{u\leq s} \big| \xi_u^i - Y_u^i \big|^2 \right]ds,
        \end{aligned}
    \end{align*} 
    where in the above inequalities and also in the following, the map $t\mapsto C(t)$ does not depend on $N$ and may change from line to line; however, it is always increasing w.r.t. its argument. Then,
    \begin{equation}\label{eq:proof_boundness_interactin_particles_sup}
        \begin{aligned}
            &\sup_{i=1,\dots,N}\mathbb{E}\left[ \sup_{0\leq s\leq t} \big| \xi_s^i - Y_s^i \big|^2 \right]\\
            &\begin{aligned}
                \,\,\leq\,&4TC(T)\int_0^t \sup_{r\leq s}\mathbb{E}\left[\left(\Big\lVert u^{\mu_N^\xi}(r,\cdot) - u_K(r,\cdot) \Big\rVert_{\infty}^2 + \Big\rVert \nabla u^{\mu_N^\xi}(r,\cdot) - \nabla u_K(r,\cdot) \Big\rVert_{\infty}^2\right)\right]ds\\
                &+8TC(T)\int_0^t \sup_{i=1,\dots,N}\mathbb{E}\left[ \sup_{u\leq s} \big| \xi_u^i - Y_u^i \big|^2 \right]ds.
            \end{aligned}
        \end{aligned}
    \end{equation}
    
    We apply Proposition \ref{prop:stability_Feynman_Kac} pathwise with $m=\mu_N^\xi(\bar{\omega})$ and $m'=\mu_N^Y(\bar{\omega})$,   Proposition \ref{prop:stability_3} with the random measures $\eta = \mu_N^Y$ and $m$, and we notice that  $\frac{1}{N}\sum_{i=1}^N\delta_{(\xi^i,Y^i)}$ is a coupling for $\mu_N$ and $\mu_N^Y$. Then, for every $t\in[0,T]$,
    \begin{align*}
        \mathbb{E}\left[ \lVert u^{\mu_N^\xi}(t,\cdot) - u_K(t,\cdot) \rVert_{\infty}^2 \right] 
        &\leq2\mathbb{E}\left[ \lVert u^{\mu_N^\xi}(t,\cdot) - u^{\mu_N^Y}(t,\cdot)\rVert_{\infty}^2\right] + 2\mathbb{E}\left[ \lVert u^{\mu_N^Y}(t,\cdot) - u_K(t,\cdot) \rVert_{\infty}^2\right]\\
        &\leq2C_1(t)D_t^2\left( \mu_N^\xi,\mu_N^Y \right) + 2C_3(t)d_2(\mu_N^Y,m)^2\\
        &\leq\frac{2C_1(t)}{N}\sum_{i=1}^N\mathbb{E}\left[ \sup_{s\leq t} \big| \xi_s^i - Y_s^i \big|^2 \right] + 2C_3(t) d_2(\mu_N^Y,m)^2\\
        &\leq2C_1(t)\sup_{i=1,...,N}\mathbb{E}\left[ \sup_{s\leq t} \big| \xi_s^i - Y_s^i \big|^2 \right] + 2C_3(t)d_2(\mu_N^Y,m)^2,
    \end{align*}
    where the homogeneous measure $d_2$ is defined in \eqref{eq:homogeneous_measure_distance}. Also, since $\mu_N^Y\in\mathbb{P}(\mathcal{P}^2(C))$ and $m\in\mathcal{P}^2\left(C \right)$ (as proved in Theorem \ref{thm:existence_uniqueness_SDE}), the above inequalities make sense. By Proposition \ref{prop:stability_Feynman_Kac:gradient} and Proposition \ref{prop:stability_3:gradient}, the same estimate holds for the second term in the right-hand side of \eqref{eq:proof_boundness_interactin_particles_sup}. 
    
    Therefore, by \eqref{eq:proof_boundness_interactin_particles_sup}
    \begin{align*}
        &\mathbb{E}\left[ \lVert u^{\mu_N^\xi}(t,\cdot) - u_K(t,\cdot) \rVert_{\infty}^2 \right] + \mathbb{E}\left[ \lVert \nabla u^{\mu_N^\xi}(t,\cdot) - \nabla u_K(t,\cdot) \rVert_{\infty}^2 \right] + \sup_{i=1,...,N}\mathbb{E}\left[ \sup_{s\leq t} \Big| \xi_s^i - Y_s^i \Big|^2 \right]\\
        &\leq\left( 2C_1(t) + 2\widetilde{C}_1(t) + 1 \right)\sup_{i=1,...,N}\mathbb{E}\left[ \sup_{s\leq t} \big| \xi_s^i - Y_s^i \big|^2 \right] +\, \left(2C_3(t)+2\widetilde{C}_3(t)\right)d_2(\mu_N^Y,m)^2\\
        &\begin{aligned}
            \,\,\leq\,&\overline{C}_1(t)\int_0^t \sup_{r\leq s}\mathbb{E}\left[\left(\Big\lVert u^{\mu_N^\xi}(r,\cdot) - u_K(r,\cdot) \Big\rVert_{\infty}^2 + \Big\rVert \nabla u^{\mu_N^\xi}(r,\cdot) - \nabla u_K(r,\cdot) \Big\rVert_{\infty}^2\right)\right]ds\\
            &+\overline{C}_1(t)\int_0^t \sup_{i=1,\dots,N}\mathbb{E}\left[ \sup_{u\leq s} \big| \xi_u^i - Y_u^i \big|^2 \right]ds +\left(2C_3(t)+2\widetilde{C}_3(t)\right)\,d_2(\mu_N^Y,m)^2,
        \end{aligned} 
    \end{align*}
    where $\overline{C}_1(t):=\left( 2C_1(t) + 2\widetilde{C}_1(t) + 1 \right)8TC(T)$ and $t\mapsto C(t)$ is an increasing map. Thus,
    \begin{align*}
        &\sup_{t\leq T}\left(\mathbb{E}\left[ \lVert u^{\mu_N^\xi}(t,\cdot) - u_K(t,\cdot) \rVert_{\infty}^2 \right] + \mathbb{E}\left[ \lVert \nabla u^{\mu_N^\xi}(t,\cdot) - \nabla u_K(t,\cdot) \rVert_{\infty}^2 \right] + \sup_{i=1,...,N}\mathbb{E}\left[ \sup_{s\leq t} \Big| \xi_s^i - Y_s^i \Big|^2 \right]\right)\\
        &\begin{aligned}
            \,\,\leq\,&\overline{C}_1(T)\int_0^T \sup_{r\leq s}\mathbb{E}\left[\left(\Big\lVert u^{\mu_N^\xi}(r,\cdot) - u_K(r,\cdot) \Big\rVert_{\infty}^2 + \Big\rVert \nabla u^{\mu_N^\xi}(r,\cdot) - \nabla u_K(r,\cdot) \Big\rVert_{\infty}^2\right)\right]ds\\
            &+\overline{C}_1(T)\int_0^T \sup_{i=1,\dots,N}\mathbb{E}\left[ \sup_{u\leq s} \big| \xi_u^i - Y_u^i \big|^2 \right]ds +\left(2C_3(T)+2\widetilde{C}_3(T)\right)\,d_2(\mu_N^Y,m)^2.
        \end{aligned} 
    \end{align*}
    By Gronwall's lemma, having taken the supremum, for every $t\in[0,T]$,
    \begin{equation*} 
        \begin{aligned}
            &\mathbb{E}\left[ \lVert u^{\mu_N^\xi}(t,\cdot) - u_K(t,\cdot) \rVert_{\infty}^2 \right] +\mathbb{E}\left[ \lVert \nabla u^{\mu_N^\xi}(t,\cdot) - \nabla u_K(t,\cdot) \rVert_{\infty}^2 \right] + \sup_{i=1,...,N}\mathbb{E}\left[ \sup_{s\leq t} \Big| \xi_s^i - Y_s^i \Big|^2 \right]\\
            &\leq C(T) d_2(\mu_N^Y,m)^2.
        \end{aligned}
    \end{equation*}
    From \eqref{eq:estimate_d2}, inequality \eqref{eq:boundness_interactin_particles} is proved.
   
    Let us turn to prove inequality (\ref{eq:boundness_interactin_particles2}).
      For every $t\in[0,T]$, we can write
    \begin{equation}\label{eq:boundness_interacting_particles_L2} 
        \mathbb{E}\left[ \lVert u^{\mu_N^\xi}(t,\cdot) - u_K(t,\cdot) \rVert_2^2 \right] \leq 2\mathbb{E}\left[ \lVert u^{\mu_N^\xi}(t,\cdot) - u^{\mu_N^Y}(t,\cdot) \rVert_2^2 \right] + 2\mathbb{E}\left[ \lVert u^{\mu_N^Y}(t,\cdot) - u_K(t,\cdot) \rVert_2^2 \right].
    \end{equation}

    From  Proposition \ref{eq:stability_2} and equation \eqref{eq:boundness_interactin_particles}, for every $t\in [0,T]$, we have that
    \begin{align*}
        \mathbb{E}\left[ \lVert u^{\mu_N^\xi}(t,\cdot) - u^{\mu_N^Y}(t,\cdot) \rVert_2^2 \right] &\leq C \mathbb{E}\left[D_t\left( \mu_N,\mu_N^Y \right)^2\right] \leq \frac{C}{N}\sum_{j=1}^N \mathbb{E}\left[ \sup_{0\leq r\leq t} \big| \xi_r^j - Y_r^j \big|^2 \right]\leq\frac{C}{N}. 
    \end{align*}

    As the second term on the right-hand side of \eqref{eq:boundness_interacting_particles_L2} regards, fixed $i\in\{1,...,N\}$, by Cauchy-Schwartz inequality, for all $t\in[0,T]$,
    \begin{equation}\label{eq:boundness_interacting_particles_L2proof_1} 
        \mathbb{E}\left[ \lVert u^{\mu_N^Y}(t,\cdot) - u_K(t,\cdot) \rVert_2^2 \right] \leq 2\mathbb{E}\left[ \lVert A(t,\cdot) \rVert_2^2 \right] + 2\mathbb{E}\left[ \lVert B(t,\cdot) \rVert_2^2 \right],
    \end{equation}
    where 
    \begin{eqnarray*}
        A(t,x) &:=&\frac{1}{N}\sum_{j=1}^N K\left(x-Y_t^j\right)\left[ V_t\left( Y^j,u^{\mu_N^Y}(Y^j) \right) - V_t\left( Y^j,u_K(Y^j) \right) \right]\,,\\
        B(t,x) &:=&\frac{1}{N}\sum_{j=1}^N K\left( x - Y_t^j \right) V_t\left( Y^j,u_K(Y^j) \right) - \mathbb{E}\left[ K\left( x -Y_t^1 \right) 
        V_t\left( Y^1,u_K(Y^1) \right) \right].
    \end{eqnarray*} 

    Now, on the one hand, from Lemma \ref{Lemma:V_continuity},
    \begin{align*}
        \big| A(t,x) \big|^2 &\leq \frac{1}{N}\sum_{j=1}^N K\left(x - Y_t^j\right)^2 \left[ V_t\left( Y^j,u^{\mu_N^Y} \right) - V_t\left( Y^j,u_K \right) \right]^2 \\
        &\leq \frac{M_K T}{N} \left( \lambda^2 c_0 \right)^2 \sum_{j=1}^N K(x - Y_t^j) \int_0^t \int_0^s \big| u^{\mu_N^Y}(r,Y_s^j) - u_K(r,Y_s^j) \big|^2 drds \\
        &\leq \frac{M_K T}{N} \left( \lambda^2 c_0 \right)^2 \sum_{j=1}^N K(x - Y_t^j) \int_0^t \int_0^s \big\lVert u^{\mu_N^Y}(r,\cdot) - u_K(r,\cdot) \big\rVert_{\infty}^2 drds.
    \end{align*}
    
    From  Proposition \ref{prop:stability_3} and \eqref{eq:estimate_d2}, we get
    \begin{align*}
        \mathbb{E}\left[ \int_{\mathbb{R}} \big| A(t,x) \big|^2\,dx \right] &\leq M_K T \left( \lambda^2 c_0 \right)^2 \int_0^t\int_0^s \mathbb{E}\left[ \lVert u^{\mu_N^Y}(r,\cdot) - u_K(r,\cdot) \rVert_{\infty}^2 \right]drds\\\
        &\leq\, M_K T^3 \left( \lambda^2 c_0 \right)^2 C_3(t)  d_2(\mu_N^Y,m)^2  \\
        &\leq\, \frac{M_K T^3 \left( \lambda^3 c_0^2 \right)^2 C_3(t)}{N}.
    \end{align*} 

    On the other hand, from the proof of \cite[Theorem 3.1]{2016b_Russo}, thanks to Lemma \ref{Lemma:V_continuity} and the assumption of $K$ being bounded by $M_K$, we have that
    \begin{equation*}
        \mathbb{E}\left[ B(t,x)^2 \right] \,\leq\, \frac{1}{N}\mathbb{E}\left[ K\left( x - Y_t^1 \right)^2 V_t\left( Y^1,u_K(Y^1) \right)^2 \right] \leq \frac{M_K}{N}\mathbb{E}\left[K\left( x - Y_t^1 \right)\right]\,.
    \end{equation*}
    
    Integrating both sides of the above inequality w.t.r. $x\in\mathbb{R}$, we obtain
    \begin{equation*} 
        \mathbb{E}\left[ \int_{\mathbb{R}} B(t,x)^2\,dx \right] \,\leq\, \frac{M_K}{N}\mathbb{E}\left[\int_{\mathbb{R}} K\left( x - Y_t^1 \right)\,dx\right] = \frac{M_K}{N}.
    \end{equation*} 
    Then, inequality \eqref{eq:boundness_interacting_particles_L2proof_1}  becomes
    \begin{equation}\label{chap4_theo_prop_chaos_Natalini_res18}
        \mathbb{E}\left[ \lVert u^{\mu_N^Y}(t,\cdot) - u_K(t,\cdot) \rVert_2^2 \right] \,\leq\, 2\,\frac{M_K}{N} + 2\,\frac{M_K T^3 \left( \lambda^3 c_0^2 \right)^2 C_3(t)}{N}.
    \end{equation}
  Therefore, from \eqref{eq:boundness_interacting_particles_L2} we get
    \begin{equation*}
        \mathbb{E}\left[ \lVert u^{\mu_N^\xi}(t,\cdot) - u_K(t,\cdot) \rVert_2^2 \right] \,\leq\, 2\,\frac{C}{N} + 4\,\frac{M_K}{N} + 4\,\frac{M_K T^3 \left( \lambda^3 c_0^2 \right)^2 C_3(t)}{N}
    \end{equation*}
  and we conclude.
\end{proof}

Proposition \ref{prop:boundness_interactin_particles} entails the propagation of chaos. In fact, it implies the following result.
\begin{theorem}
   Let $N\in \mathbb N$. Let $(Y^1,\ldots,Y^N)$ be the solution to the system \eqref{eq:system_particles_iid} with $u^m_K$, and $(\xi^1,\ldots,\xi^N)$ be the solution to the system  \eqref{eq:system_particles_interacting} with $u^{\mu_N^\xi}$ associated to the empirical measure $\mu_N^\xi$ given by \eqref{eq:empirical_measure_mu_N_xi}. Then propagation of chaos occurs, according with Definition \ref{def:chaos_propagation}.
\end{theorem}
\begin{proof}
    For Proposition \ref{prop:boundness_interactin_particles}, we have that, said $L^2=L^2\left( \Omega,\mathcal{F},\mathbb{P} \right)$,\begin{equation*}
        \left( \xi^{1,N}- Y^1\,,\; \xi^{2,N}- Y^2\,,\; \ldots, \xi^{k,N}- Y^k  \right) \xrightarrow[L^2]{N\to \infty}0.
    \end{equation*}
    In particular, this implies the convergence in law of $\left( \xi^{1,N} , \ldots,  \xi^{k,N} \right)$ to the random vector $ \left(  Y^1   \ldots,   Y^k  \right)$.
    The propagation of chaos property is established from the independence and the identical distribution of the  $\{Y^i\}_{i=1}^N$.
\end{proof}

This result rigorously justifies the microscopic interpretation introduced above. Indeed, it shows that, in the limit as $N\rightarrow\infty$, the interacting particle system behaves asymptotically as a family of independent copies of the nonlinear process $Y$. In particular, the empirical (weighted) distribution of the particle system converges to the regularised solution of the limiting McKean--Feynman--Kac equation.

\appendix
\section{Appendix}\label{appendix}
This appendix provides the proofs of several technical results employed in the previous sections.
\subsection{Proof of Proposition \ref{prop:stability_Feynman_Kac}}\label{append_prop_eq:stability_1}
    Let $r,s,t\in[0,T]$, $r<s<t$. Fix two measures $m,m'\in\mathcal{P}^2\left(C \right)$ and $y,y'\in\mathbb{R}$. In the following, $u^m_K$ and $u^{m^\prime}_K$ will denote the respective solutions of \eqref{eq:Feynman_Kac_generic_measure}.  

    Since $(a+b)^2\le 2a^2+2b^2$, we have that
    \begin{equation}\label{eq:AppendixA_diff^2}
        \big| u^m_K(t,y) - u^{m^\prime}_K(t,y') \big|^2 \,\leq\, 2\big| u^m_K(t,y) - u^m_K(t,y') \big|^2 + 2\big| u^m_K(t,y') - u^{m^\prime}_K(t,y') \big|^2\,.
    \end{equation}
   By i) of Lemma \ref{Lemma:V_continuity} and assumption \eqref{eq:K_LK},
    \begin{equation*}
        \begin{aligned}
            &\big| u^m_K(t,y) - u^m_K(t,y') \big|\\
            &=\Bigg| \int_{C }K\big( y - X_t(\omega) \big)V_t\big(u(\cdot,X_\cdot(\omega))\big)m(d\omega) - \int_{C }K\big( y' - X_t(\omega) \big)V_t\big(u(\cdot,X_\cdot(\omega))\big)m(d\omega) \Bigg|\\
            &\leq\, \int_{C }\Big| K\big( y - X_t(\omega) \big) - K\big( y' - X_t(\omega) \big) \Big|\cdot\Big| V_t\big(u(\cdot,X_\cdot(\omega))\big) \Big|m(d\omega)\leq\, L_K\big| y-y' \big|.
        \end{aligned}
    \end{equation*}
    As for the second term in \eqref{eq:AppendixA_diff^2}, given $\pi\in\Pi(m,m')$, we get
    \begin{equation}\label{eq:appendixA_2nd_term}
        \begin{aligned}
            &\big| u^m_K(t,y') - u^{m^\prime}_K(t,y') \big|\\
            &\leq \int_{C \times C }\Big| K\big( y' - X_t(\omega) \big)V_t\big(u(\cdot,X_\cdot(\omega))\big)- K\big( y' - X_t(\omega') \big)V_t\big(u(\cdot,X_\cdot(\omega'))\big) \Big|\pi\left(d\omega,d\omega'\right).
        \end{aligned}
    \end{equation}
   Let  $x,x'\in C $, and $z,z^\prime\in C ^+$. By Lemma \ref{Lemma:V_continuity} and  $K$ being bounded,
    \begin{equation*}
        \begin{aligned}
            &\Big| K(y'-x_t)V_t(z) - K(y'-x_t')V_t(z^\prime) \Big|\\
            &\leq\Big| K(y'-x_t) - K(y'-x_t') \Big|\cdot\Big|V_t(z)\Big|+\Big| V_t(z) - V_t(z^\prime)\Big|\cdot\Big|K(y'-x_t')\Big|\\
            &\leq\,L_K\big|x_t - x_t'\big| + \lambda^2c_0M_K\int_0^t\int_0^s\big|z_r - z_r'\big|drds.
        \end{aligned}
    \end{equation*}
  Therefore, 
    \begin{align}
        &\big| u^m_K(t,y) - u^{m^\prime}_K(t,y^\prime) \big|^2\notag&&\\
        &\begin{aligned}
            \,\,\leq\, &2L_K^2\big| y - y' \big|^2 + 4L_K^2 \int_{C \times C } \sup_{s\leq t} \Big| X_s(\omega) - X_s(\omega') \Big|^2\pi\left(d\omega,d\omega'\right)\\
            &+4\left(\lambda^2c_0M_K\right)^2 \int_{C \times C }\int_0^t\int_0^s \Big| u(r,X_s(\omega)) - u(r,X_s(\omega')) \Big|^2dr ds\pi\left(d\omega,d\omega'\right)
        \end{aligned}\notag&&\\
        &\begin{aligned}
            \,\,\leq\, 2C\Bigg(&\big| y - y' \big|^2 + \int_{C \times C } \sup_{s\leq t} \Big| X_s(\omega) - X_s(\omega') \Big|^2\pi\left(d\omega,d\omega'\right)\\
            &+ \int_{C \times C }\int_0^t\sup_{r\leq s} \Big| u(r,X_s(\omega)) - u(r,X_s(\omega')) \Big|^2 ds\pi\left(d\omega,d\omega'\right)\Bigg),
        \end{aligned}\label{res1_primo_risultato_stability_u_m_Natalini}&&
    \end{align}
    where $C:=2L_K^2 + 2T(\lambda^2c_0M_K)^2$.
    
    To estimate the third term in the right-hand side of \eqref{res1_primo_risultato_stability_u_m_Natalini},
     we use inequality \eqref{res1_primo_risultato_stability_u_m_Natalini} with $r$ instead of $t$, $y=X_s(\omega)$, and $y'=X_s(\omega')$. That is, 
    \begin{align}
        &\sup_{r\leq s}\big| u^m_K\left(r,X_s(\omega)\right) - u^{m^\prime}_K\left(r,X_s(\omega')\right) \big|^2\notag&&\\
        &\begin{aligned}
           \leq 2C\Bigg(&\sup_{s\leq t}\Big| X_s(\omega) - X_s(\omega') \Big|^2 + \int_{C \times C } \sup_{s\leq t} \Big| X_s(\omega) - X_s(\omega') \Big|^2\pi\left(d\omega,d\omega'\right)\\
            &+ \int_{C \times C }\int_0^s\sup_{\tau\leq r} \Big| u(\tau,X_r(\omega)) - u(\tau,X_r(\omega')) \Big|^2 dr\pi\left(d\omega,d\omega'\right)\Bigg).
        \end{aligned}\label{res2_primo_risultato_stability_u_m_Natalini}&&
    \end{align}
    By integrating both sides of \eqref{res2_primo_risultato_stability_u_m_Natalini} over $C \times C $ w.r.t. $\pi$ and over $[0,t]$ w.r.t. $s$,  we get
    \begin{align*}
        &\int_{C\times C}\int_0^t\sup_{r\leq s}\big| u^m_K\left(r,X_s(\omega)\right) - u^{m^\prime}_K\left(r,X_s(\omega')\right) \big|^2\pi(d\omega,d\omega^\prime)&&\\
        &\begin{aligned}
           \leq 2C\Bigg(&2T\int_{C \times C } \sup_{s\leq t} \Big| X_s(\omega) - X_s(\omega') \Big|^2\pi\left(d\omega,d\omega'\right)\\
            &+ \int_{C \times C }\int_0^t\int_0^s\sup_{\tau\leq r} \Big| u(\tau,X_r(\omega)) - u(\tau,X_r(\omega')) \Big|^2 dr\pi\left(d\omega,d\omega'\right)\Bigg).
        \end{aligned}\label{res2_primo_risultato_stability_u_m_Natalini}&&
    \end{align*}
    By Gronwall's lemma,
    \begin{equation}\label{res3_primo_risultato_stability_u_m_Natalini}
        \begin{aligned}
            &\int_{C\times C}\int_0^t\sup_{r\leq s}\big| u^m_K\left(r,X_s(\omega)\right) - u^{m^\prime}_K\left(r,X_s(\omega')\right) \big|^2\pi(d\omega,d\omega^\prime)\\
            &\leq 4CTe^{2CTt}\int_{C \times C }\sup_{s\leq t}\Big| X_s(\omega) - X_s(\omega') \Big|^2\pi\left(d\omega,d\omega'\right).
        \end{aligned}
    \end{equation}
    Substituting \eqref{res3_primo_risultato_stability_u_m_Natalini} in \eqref{res1_primo_risultato_stability_u_m_Natalini}, we obtain
    \begin{align*}
        &\big| u^m_K(t,y) - u^{m^\prime}_K(t,y) \big|^2&&\\
        &\leq\, 8C^2Te^{2CTt}\Bigg(\big| y - y' \big|^2 + \int_{C \times C } \sup_{s\leq t} \Big| X_s(\omega) - X_s(\omega') \Big|^2\pi\left(d\omega,d\omega'\right)\Bigg).&&
    \end{align*}

    Taking the infimum over $\Pi(m,m')$, we get
    \begin{equation*}
        \big| u^m_K(t,y) - u^{m^\prime}_K(t,y) \big|^2 \,\leq\, C_1(t)\cdot\left(\big|y - y'\big|^2 + D_t(m,m')^2\right),
    \end{equation*}
    where $C_1(t):=8C^2Te^{2CTt}$ and $C:=2L_K^2 + 2(\lambda^2c_0M_K)^2$, and we conclude.
    
    This result implies that the application is continuous with respect to $(m,y)$ uniformly with respect to time. Therefore, it remains to show that the map $t\mapsto u^m_K(t,y)$ is continuous for fixed $(m,y)\in\mathcal{P}\left(C \right)\times\mathbb{R}$. By Lemma \ref{Lemma:Lambda} and Lemma \ref{Lemma:V_continuity}, the second statement of the proposition can be proved using the same arguments as in \cite[Proposition 3.2]{2016_Russo}.

\subsection{Proof of Proposition \ref{eq:stability_2}}\label{append_prop_eq:stability_2}
    Let $m,m'\in\mathcal{P}^2\left( C  \right)$. Since $K\in L^2\left(\mathbb{R}\right)$ and $V$ is bounded (see Lemma \ref{Lemma:V_continuity}), using Jensen's inequality, it is immediate to show that the maps $x\mapsto u^m_K(r,x)$ and $x\mapsto u^{m^\prime}_K(r,x)$ belong to $L^2\left(\mathbb{R}\right)$, for every $r\in[0,T]$. Therefore, for any $\pi\in\Pi(m,m')$, we can compute
    \begin{align}
        \lVert u^m_K(t,\cdot) - u^{m^\prime}_K(t,\cdot) \rVert_2^2
        &= \int_{\mathbb{R}} |u^m_K(t,y) - u^{m^\prime}_K(t,y)|^2 \, dy \notag\\
        &\begin{aligned}[t]
           \leq \int_{C \times C} \int_{\mathbb{R}} 
           \Big| &K\!\left( y - X_t(\omega) \right) V_t\!\left( u^m_K(X(\omega)) \right) \\
                 &- K\!\left( y - X_t(\omega') \right) V_t\!\left( u^{m^\prime}_K(X(\omega')) \right) \Big|^2 \,
           dy \, \pi\!\left(d\omega, d\omega'\right).
        \end{aligned}\label{eq:proof_stability1}
    \end{align}

    By Lemma \ref{Lemma:V_continuity} and Definition \ref{def:kernel} of the mollifier $K$, for any $x,x'\in C $ and $z,z^\prime\in C ^+$,
    \begin{align*}
        &\int_{\mathbb{R}} \Big| K(y-x_t)V_t(z) - K(y-x_t')V_t(z^\prime) \Big|^2dy\\
        &\leq 2\int_{\mathbb{R}} \Big| K(y-x_t) - K(y-x_t') \Big|^2\cdot\Big|V_t(z)\Big|^2dy + 2\int_{\mathbb{R}} \Big| V_t(z) - V_t(z^\prime) \Big|^2\cdot\Big|K(y-x_t')\Big|^2dy\\
        &\begin{aligned}[t]
            \,\leq\,&2\int_{\mathbb{R}} \Big| K(y-x_t) - K(y-x_t') \Big|^2dy\\
            &+ 2\left((\lambda^2c_0)^2\int_0^t\int_0^s\big| z_r - z_r' \big|^2drds\right)\cdot M_K\int_{\mathbb{R}} K(y-x_t')dy
        \end{aligned}\\
        &\leq 2\int_{\mathbb{R}} \Big| K(y-x_t) - K(y-x_t') \Big|^2dy + 2M_K(\lambda^2c_0)^2\int_0^t\int_0^s\big| z_r - z_r' \big|^2drds\\
        &\leq 2(M_K^\prime)^2 \big| x_t - x_t' \big|^2 + 2M_K(\lambda^2c_0)^2\int_0^t\int_0^s\big| z_r - z_r' \big|^2drds,
    \end{align*}
    where the first addend of the last inequality is due to Plancherel's theorem.

    Then, \eqref{eq:proof_stability1} becomes
    \begin{align*}
        &\lVert u^m_K(t,\cdot) - u^{m^\prime}_K(t,\cdot) \rVert_2^2\\
        &\begin{aligned}[t]
            \,\leq\,C_2\Bigg(&\int_{C \times C } \sup_{s\leq t} \Big| X_s(\omega) - X_s(\omega') \Big|^2\pi\left(d\omega,d\omega'\right)\\
            &+\int_{C \times C }\int_0^t\int_0^s \Big|u(r,X_s(\omega)) - u(r,X_s(\omega')) \Big|^2\pi\left(d\omega,d\omega'\right)drds\Bigg)
        \end{aligned}\\
        &\leq\,C_2\left(1+C_1(t)\right)\int_{C \times C } \sup_{s\leq t} \Big| X_s(\omega)-X_s(\omega') \Big|^2\pi\left(d\omega,d\omega'\right),
    \end{align*}
    where $C_2:=2(M_K^\prime)^2+2M_K(\lambda^2c_0)^2$, and the second inequality is obtained by doing the same computations as in Proposition \ref{prop:stability_Feynman_Kac}. By taking the infimum over $\Pi(\omega,\omega')$,
    \begin{equation*}
        \lVert u^m_K(t,\cdot) - u^{m^\prime}_K(t,\cdot) \rVert_2^2 \,\leq\, C_2\Big( 1+C_1(t) \Big)D_t(m,m')^2
    \end{equation*}
    and we conclude.

\subsection{Proof of Lemma \ref{lemma:continuity_b}}\label{section:continuity_b}
    Again, we consider the integral  notation for $u^m_K$, i.e. $u^m_K(\cdot,y)(s):=\int_0^s u^m_K(r,y)dr$, and denote by $\varphi^m(t,y)=\varphi_0 + \varphi_1 c_0 \exp\left(-\lambda u^m_K(\cdot,y)(t)\right)\in [\varphi_0,\varphi_0+\varphi_1 c_0],$ for a measure $\mu\in \mathcal{P}(C )$. Now,
    \begin{equation*}
    \begin{split}
     &\Big| b\Big(u^m_K(\cdot,y)(t),\nabla u^m_K(\cdot,y)(t)\Big) - b\Big(u^{m^\prime}_K(\cdot,y)(t),\nabla u^{m^\prime}_K(\cdot,y)(t)\Big) \Big|\\
     &=\varphi_1\lambda c_0 \Bigg| \frac{e^{-\lambda u^m_K(\cdot,y)(t)}\nabla u^m_K(\cdot,y)(t)}{\varphi^m(t,y)}-\frac{e^{-\lambda u^{m^\prime}_K(\cdot,y)(t)}\nabla u^{m^\prime}_K(\cdot,y)(t)}{\varphi^{m^\prime}(t,y)} \Bigg|\\
          &\le\, \frac{\varphi_1\lambda c_0 }{\varphi_0^2}\Bigg| {e^{-\lambda u^m_K(\cdot,y)(t)}\varphi^{m^\prime}(t,y)\nabla u^m_K(\cdot,y)(t)}  - {e^{-\lambda u^{m^\prime}_K(\cdot,y)(t)}\varphi^{m}(t,y)\nabla u^{m^\prime}_K(\cdot,y)(t)}  \Bigg|\\
     &\le  \frac{\varphi_1\lambda c_0 }{\varphi_0^2}\Bigg|\varphi_0 {e^{-\lambda u^m_K(\cdot,y)(t)} \nabla u^m_K(\cdot,y)(t)} -\varphi_0 
     e^{-\lambda u^{m^\prime}_K(\cdot,y)(t) }  \nabla u^{m^\prime}_K(\cdot,y)(t)\\
     &  \hspace{2cm} +  
     \varphi_1 c_0 \,e^{-\lambda u^m(\cdot,y)(t)} 
                   e^{-\lambda u^{m^\prime}_K(\cdot,y)(t) }
     \left(
     \nabla u^m(\cdot,y)(t) -\nabla u^{m^\prime}_K(\cdot,y)(t)  \right) \Bigg|.
      \end{split}
    \end{equation*}
    Since  $e^{-x}\leq 1+x$ for any $ x\geq 0$, by \eqref{Lemma:exp(u^m_K)_bounded_1} and \eqref{Lemma:exp(u^m_K)_bounded_3},
    \begin{align*}
        &\left| b\big(t,u^m(\cdot,y)(t)\big) - b\big(t,u^{m^\prime}_K(\cdot,y)(t)\big) \right|\\
        &\begin{aligned}
            \,\leq\, &\frac{\varphi_1\lambda c_0}{\varphi_0}\Big|e^{-\lambda u^m_K(\cdot,y)(t)}\nabla u^m_K(\cdot,y)(t) - e^{-\lambda u^{m^\prime}_K(\cdot,y)(t)}\nabla u^{m^\prime}_K(\cdot,y)(t)\Big|\\
            &+ \frac{\varphi_1^2\lambda c_0^2}{\varphi_0^2}\Big|\nabla u^m_K(\cdot,y)(t) - \nabla u^{m^\prime}_K(\cdot,y)(t)\Big|
        \end{aligned}\\
        &\begin{aligned}
            \,\leq\, &\frac{\varphi_1\lambda c_0}{\varphi_0}e^{-\lambda u^m_K(\cdot,y)(t)}\Big|\nabla u^m_K(\cdot,y)(t) - \nabla u^{m^\prime}_K(\cdot,y)(t)\Big|\\
            &+ \frac{\varphi_1\lambda c_0}{\varphi_0}\Big|\nabla u^{m^\prime}_K(\cdot,y)(t)\Big|\Big|e^{-\lambda u^m_K(\cdot,y)(t)} - e^{-\lambda u^{m^\prime}_K(\cdot,y)(t)}\Big|\\
            &+\frac{\varphi_1^2\lambda c_0^2}{\varphi_0^2}\Big|\nabla u^m_K(\cdot,y)(t) - \nabla u^{m^\prime}_K(\cdot,y)(t)\Big|
        \end{aligned}\\
        &\leq \widetilde{C} \Big|u^m_K(\cdot,y)(t) - u^{m^\prime}_K(\cdot,y)(t)\Big| + \Big|\nabla u^m_K(\cdot,y)(t) - \nabla u^{m^\prime}_K(\cdot,y)(t)\Big|,
    \end{align*}
     where $\widetilde C= \max \left( {\varphi_1\lambda c_0}/{\varphi_0}, {\varphi_1\lambda c_0 M_K^\prime}/{\varphi_0},{\varphi_1^2\lambda c_0^2}/{\varphi_0^2} \right). $

\paragraph{Acknowledgments.} The research is carried out within the research project PON 2021(DM 1061, DM 1062) ``Deterministic and stochastic mathematical modelling and data analysis within the study for the indoor and outdoor impact of the climate and environmental changes for the degradation of the Cultural Heritage" of the Università degli Studi di Milano.
  

D.M. and S.U. 
  are members of GNAMPA (Gruppo Nazionale per l’Analisi Matematica, la Probabilità e le loro Applicazioni) of the Italian Istituto Nazionale di Alta Matematica (INdAM).
\printbibliography
\end{document}